\documentclass[preprint,1p,times,twocolumn,letterpaper]{elsarticle}
\usepackage{amsmath,amssymb,amsfonts,amsthm} 
\usepackage[pdfstartview=FitH,colorlinks=true,bookmarks=false]{hyperref}
\usepackage{subfigure}
\usepackage{color}
\usepackage{titlesec}
\usepackage{xr}
\usepackage{xspace}

\newtheorem{thm}{Theorem}[section]

\newtheorem{lem}[thm]{Lemma}

\theoremstyle{definition}

\theoremstyle{definition}

\theoremstyle{definition}
\newtheorem{defn}[thm]{Definition}

\journal{DCDS-S}

\renewcommand{\leq}{\leqslant}
\renewcommand{\geq}{\geqslant}
\newcommand{\cC}{C}

\newcommand{\R}{\mathbb{R}}
\newcommand{\cone}{\mathop \Sigma _{\Delta}}
\newcommand{\wedg}{\mathop \Sigma_{w}}
\newcommand{\cusp}{\mathop \Sigma_{c}}
\newcommand{\FigureCaption}{\caption}

\renewcommand{\phi}{\varphi}
\renewcommand{\epsilon}{\varepsilon}

\newcommand{\cM}{\mathcal{M}}
\newcommand{\cN}{\mathcal{N}}
\newcommand{\cQ}{\mathcal{J}}
\newcommand{\udot}{\dot{u}}

\newcommand{\thmlocalcone}{Theorem 4.1 in \cite{MH:2016}\xspace}

\begin{document}

\begin{frontmatter}

\title{Generalised Lyapunov-Razumikhin techniques for scalar state-dependent delay differential equations}

\author[Manitoba]{F.M.G. Magpantay}
\ead{felicia.magpantay@umanitoba.ca}
\author[McGill]{A.R. Humphries}
\ead{tony.humphries@mcgill.ca}

\address[Manitoba]{Department of Mathematics, University of Manitoba, 186 Dysart Road, Winnipeg, MB, Canada R3T 2N2}
\address[McGill]{Department of Mathematics and Statistics, McGill University, 805 Sherbrooke St. W., Montreal, QC, Canada H3A 0B9}

\begin{abstract}
We present generalised Lyapunov-Razumikhin techniques for establishing global asymptotic stability of steady-state solutions of scalar delay differential equations.
When global asymptotic stability cannot be established, the technique can be used to derive bounds on the persistent dynamics.
The method is applicable to constant and variable delay problems, and we illustrate the method by applying it to the state-dependent delay differential equation known as the sawtooth equation, to find parameter regions for which the steady-state solution is globally asymptotically stable. We also establish bounds on the
periodic orbits that arise when the steady-state is unstable.
This technique can be readily extended to apply to other scalar delay differential equations with negative feedback.
\end{abstract}

\end{frontmatter}

\section{Introduction}
\label{Secn:Intro}
Delay differential equations (DDEs) are differential equations for which the rate of change of the state at time $t$ depends on its value at some past time $\alpha(t)\leq t$.
When $\tau$ also depends on the state of the solution, these systems are called state-dependent delay differential equations (SDDEs). A general SDDE with one discrete delay and an initial function $\varphi$ that gives the value of the state $u$ before the initial time of $t_0$, can be written as
\begin{equation}
\left. \begin{array}{ll}
\dot u(t)=f\big(t,u(t),u(\alpha(t,u(t))\big), & t> t_0,\\
u(t) = \varphi(t), & t\leqslant t_0,
\end{array}\right\}
\end{equation}
where $\alpha(t,u(t))\leq t$.
Such systems naturally arise in many applications including electrodynamics~\cite{Driver:1963}, population dynamics~\cite{MagpantayKosovalicWu:2014}, automatic position control \cite{Walther:2003} and milling~\cite{Insperger:2007}.
When the delays can be bounded~\textit{a priori}, SDDEs can be treated as retarded functional differential equations (RFDEs).  There is a well-established theory for treating RFDEs using infinite-dimensional dynamical systems on function spaces \cite{DGVLW95,GW13,Hale:1,HaleLunel:1},
however this theory is generally not directly applicable to state-dependent delay problems.
A more complete discussion of relatively recent progress in state-dependent delay equation theory and applications is available in \cite{HartungKrisztinWaltherWu:1}.


In this paper we work directly in the SDDE framework to derive a sufficient condition for global asymptotic stability of the equilibria of scalar SDDEs.
We also derive bounds to any persistent dynamics (such as periodic solutions) of the system when the equilibrium is not guaranteed to be globally asymptotically stable.

Lyapunov-Razumikhin techniques were first developed by
Razumikhin~\cite{Razumikhin:1}, who realised that to establish stability instead of defining a Lyapunov functional that decreased monotonically along the solution, it is sufficient just to consider the case where the solution is about to exit a ball centered at the steady state. Razumikhin-type results for both Lyapunov and asymptotic stability have since been presented by many authors. A detailed presentation can be found in Chapter 5 of \cite{HaleLunel:1}. Many of the results and examples are tailored to constant delay problems, but amongst others
\cite{Krisztin:3,Myshkis:1,Yorke:1} include results for time-dependent delays. In \cite{MH:2016} we present Razumikhin-type Lyapunov and asymptotic stability results tailored to SDDEs.

In this paper we generalise Lyapunov-Razumikhin techniques. 
Instead of considering the solution when it is about to exit some ball, we will consider oscillatory solutions when they achieve a local extremum to derive a sequence of bounds on the solution as $t\to\infty$.
We apply this generalised Lyapunov-Razumikhin technique to the model SDDE
\begin{equation}
\left. \begin{array}{ll}
\dot u(t)=\mu u(t)+\sigma u(t-a-cu(t)), & t> 0,\\
u(t) = \varphi(t), & t\leqslant 0.
\end{array}\right\}
\label{Eqn:1Delay}
\end{equation}
This model problem has a known parameter region $\Sigma_\star=\cone\cup\wedg\cup\cusp$ (see Definition~\ref{Defn:StabilityRegion}) in which the zero solution has been proven to be locally exponentially stable~\cite{GyoriHartung:1}. We prove global asymptotic stability of the steady-state in part of $\wedg$ in the delay-independent stability region. We demonstrate numerically that the steady state cannot be globally asymptotically stable in all of $\wedg$ because of the existence of a periodic orbit generated by a sub-critical Hopf bifurcation along part of the boundary of the stability region.
We also derive bounds for any
periodic solutions or other persistent dynamics of \eqref{Eqn:1Delay}
for parameter values outside the region where we can show global asymptotic stability, by fnding an absorbing set for the dynamics.

The constant delay case ($c=0$) of the model state-dependent DDE \eqref{Eqn:1Delay} is known as Hayes equation~\cite{Hayes50}.
This is the standard model problem used to illustrate stability theory for constant delay DDEs~\cite{HaleLunel:1,IS11,Smith:1}, and is also the test problem for stability of numerical DDE methods~\cite{BellenZennaro:1}.
The state-dependent case ($c\ne 0$) of \eqref{Eqn:1Delay} is the natural generalisation of Hayes equation to state-dependent DDEs.
This was introduced by Mallet-Paret and Nussbaum~\cite{MalletParetNussbaum:1}.
In contrast to Hayes equation which is linear, the state-dependent case of \eqref{Eqn:1Delay} is nonlinear and, in parts of the parameter space, has bounded periodic solutions.
Mallet-Paret and Nussbaum investigate the existence and form of the slowly oscillating periodic solutions of a singularly
perturbed version of \eqref{Eqn:1Delay} in detail in \cite{MalletParetNussbaum:1, MalletParetNussbaum:2, MalletParetNussbaum:3,MalletParetNussbaum:4}.
This DDE is also known as the sawtooth equation and has also been studied in \cite{HBCHM:1,Humphries:1,MH:2016,KE:1}.


This paper is divided into four sections.
In Section~\ref{Secn:Properties} we review existing stability results for the model SDDE~\eqref{Eqn:1Delay}.
In Section~\ref{Secn:DerivingBounds} we consider the model SDDE~\eqref{Eqn:1Delay} with $\mu<0$ when an \textit{a priori} bound can be established on solutions.
We show this implies that all solutions either oscillate forever or are eventually monotonic and converge to the steady state.
We then apply our generalised Lyapunov-Razumikhin technique to this equation by considering how solutions of \eqref{Eqn:1Delay} behave at local extrema to obtain recursive bounds on the solutions in Theorem~\ref{Thm:1Delay_changingbounds}.
This leads to numerically verifiable conditions for the global asymptotic stability of the steady state of \eqref{Eqn:1Delay} which are given in Theorem~\ref{Thm:GAS}.
This establishes global asymptotic stability of the steady-state solution in most of the region where we previously~\cite{MH:2016} showed local asymptotic stability (the regions are plotted and compared in Figure~\ref{Fig:GAS}(a)). We also demonstrate that the steady state is only locally asymptotically stable in part of this parameter region because of a subcritical Hopf bifurcation (see Figure~\ref{Fig:GAS}(b)).
In regions where global asymptotic stability cannot be shown, Theorem~\ref{Thm:1Delay_changingbounds} can be used to find improved bounds on periodic solutions and other persistent dynamics.
This technique is illustrated in Figure~\ref{Fig:Bounds} and the improvement on the bounds is shown in Figure~\ref{Fig:BoundsPlane}.
In Section~\ref{Secn:Conclusions} we present a brief summary of our results and comment on how they can be extended to
other DDEs.

\section{Model equation properties}
\label{Secn:Properties}

In this section we summarise known properties of the model SDDE \eqref{Eqn:1Delay}. We require $a>0$
so that the delay term $a+cu$ is positive at the trivial steady-state $u=0$.
Since the SDDE \eqref{Eqn:1Delay} is invariant under the transformation $u\mapsto -u$, $c\mapsto -c$, we consider only the case $c>0$ (there is no state-dependency and no nonlinearity if $c=0$). We also assume that $\mu+\sigma<0$, which ensures that the deviating argument $\alpha(t,u(t))=t-a-cu(t)\leq t$, i.e. the delayed term does not become an 
advance~\cite{MH:2016}.

In this work we restrict attention to the case where $\mu<0$, which in \cite{MH:2016} was shown to ensure that all solutions of the SDDE remain bounded; this is a necessary property for the techniques that we will deploy. This result is stated as Theorem~\ref{Thm:1DelayProperties} below. It is useful to define the following constants,
\begin{equation} \label{Eqn:1Delay:L0M0tau0}
M_0=-\frac{a}{c}, \qquad N_0=\frac{a\sigma}{c\mu}, \qquad \tau_0=a+cN_0, \qquad \tau=a+cN,
\end{equation}
(where $N$ is defined in the following theorem).

\begin{thm} [Existence and boundedness of solutions to the model SDDE] \label{Thm:1DelayProperties}
Let $a>0$, $c>0$ and $\mu<0$.
We also make the following assumptions:
\begin{enumerate}
\item If $0<\sigma<-\mu$ let $ M\in [M_0,0)$, $ N>0$,
and suppose that $ \phi(t)\in [M,N]$ for all $t\in[-\tau,0]$.
\item If $\sigma\leq0$ let $N=\max\{N_0,\phi(0)\}$, and suppose $\phi(t)\geq M_0$ for all $t\in[-\tau,0]$.
\end{enumerate}
Let the initial history function $\phi$ be continuous.
Then, there exists at least one solution $u\in C^1([0,\infty),\R)$ which solves \eqref{Eqn:1Delay} for all $t\geq0$.
Furthermore, any solution satisfies the bounds,
\begin{equation} \label{eq:ubd}
 u(t)\in (M,N),\quad \forall t>0.
\end{equation}
and
\begin{equation} \label{Eqn:alphabounds}
\alpha(t,u(t))=t-a-cu(t)\in(t-\tau,t),\quad \forall t>0.
\end{equation}
If $\phi$ is locally Lipschitz the solution is unique.
\end{thm}
\begin{proof}
This is a special case ($c>0$ and $\mu<0$) of Theorem 3.3 in \cite{MH:2016}.
\end{proof}

We denote by $\Sigma_\star$ the region of $(\mu,\sigma)$ parameter values for which the steady state solution $u=0$ is asymptotically stable. This is well-known in the constant delay case ($c=0$) and derived from the characteristic equation of \eqref{Eqn:1Delay} \cite{ElsgoltsNorkin:1}.
Gy\"ori and Hartung \cite{GyoriHartung:1} showed that the state-dependent case ($c\ne 0$) of \eqref{Eqn:1Delay} has the same (exponentially) asymptotic stability region, that is the interior of the parameter region is the same in the state-dependent case. On the boundary of $\Sigma_\star$ the steady-state is Lyapunov stable for the constant delay case, and the stability is delicate in the state-dependent case \cite{Stumpf:1}.

\begin{figure}[th]
\begin{center}
\subfigure[Stability region]{\includegraphics[height=2.4in]{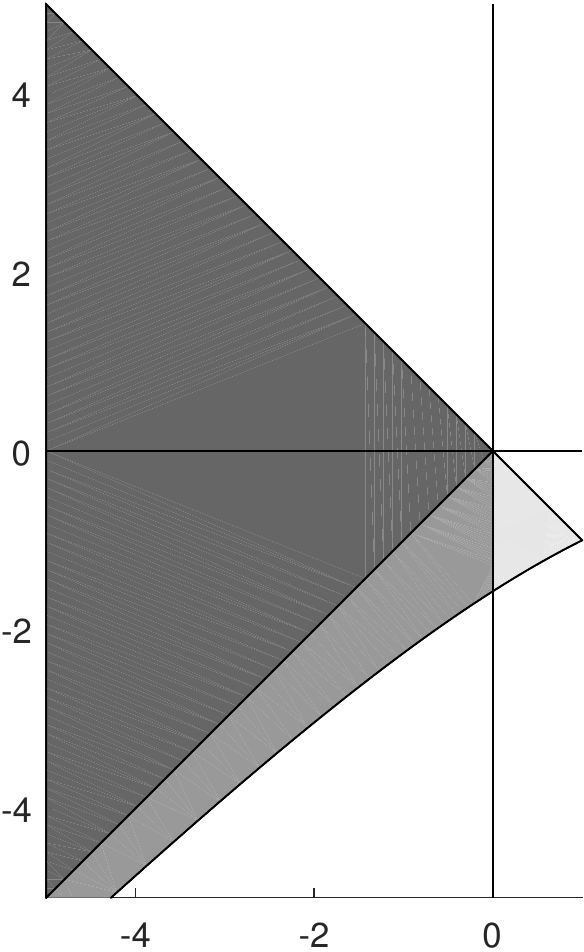}
\put(-7,0){{\footnotesize $\mu$}}
\put(-110,135){{\footnotesize $\sigma$}}
\put(-80,120){{\footnotesize $\cone$}}
\put(-11,75){{\footnotesize $\cusp$}}
\put(-40,60){{\footnotesize $\wedg$}}
}
\hspace{0.2in}
\subfigure[Sample solutions]{\includegraphics[height=2.4in]{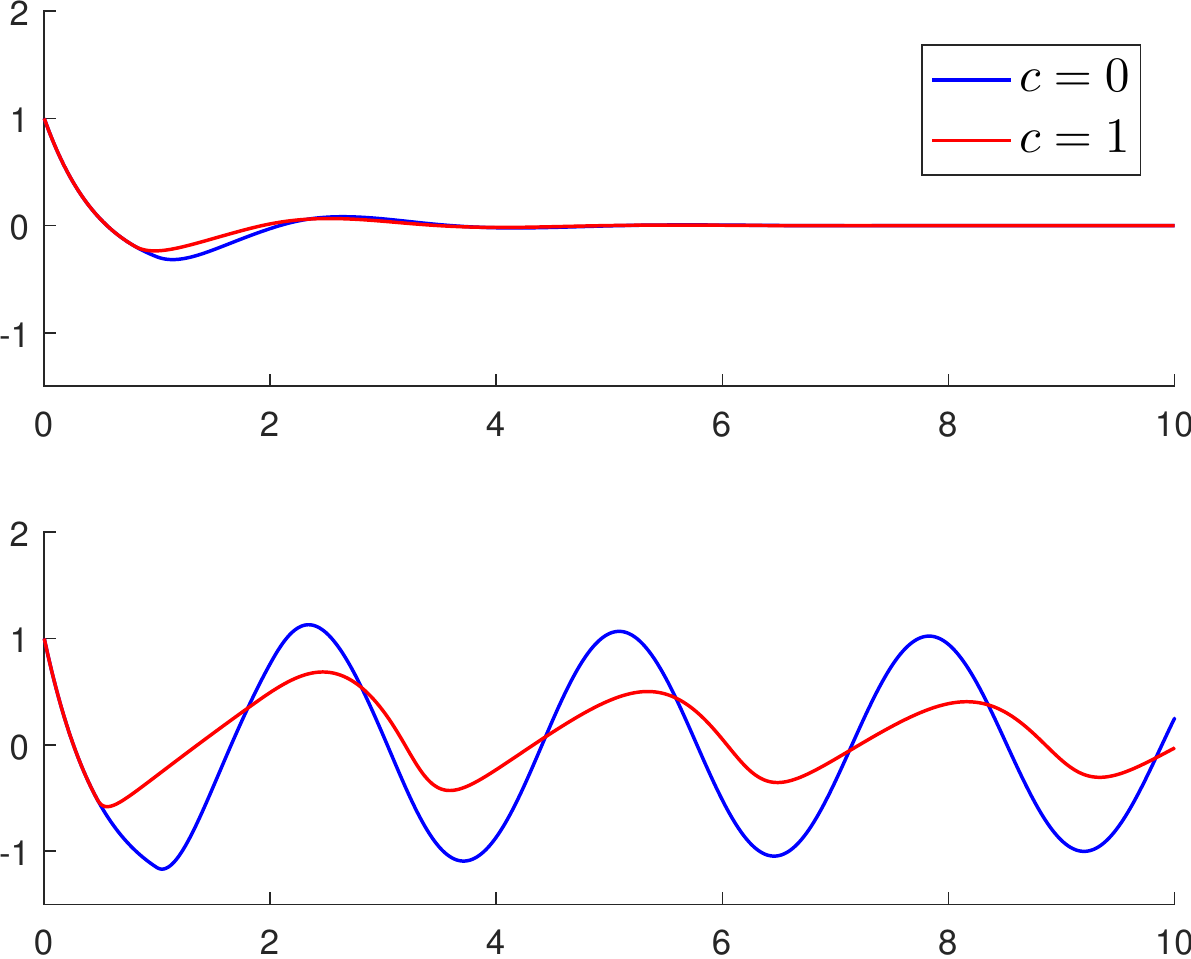}
\put(-25,-1){{\footnotesize$t$}}
\put(-25,92){{\footnotesize$t$}}
\put(-225,158){{\footnotesize$u(t)$}}
\put(-225,64){{\footnotesize$u(t)$}}
\put(-185,165){\small $\dot{u}(t)=-2u(t)-u(t-1-cu(t))$}
\put(-170,70){\small $\dot{u}(t)=-2u(t)-3u(t-1-cu(t))$}
}
\FigureCaption{\label{Fig:ConstantDelay_divisions}
(a) The analytic stability region $\Sigma_\star$ in the $(\mu,\sigma)$ plane, divided into the delay-independent cone $\cone$, and the delay-dependent wedge $\wedg$ and cusp $\cusp$. (b) Sample dynamics using parameter values in the $\cone$ (upper panel) and $\wedg$ (lower panel). Both state-dependent ($c=1$) and constant delay ($c=0$) solutions are shown in each case. The initial function is the constant function $\varphi(t)=1$ for all the examples.}
\end{center}
\end{figure}

\begin{figure}[th]
\begin{center}
\includegraphics[scale=1]{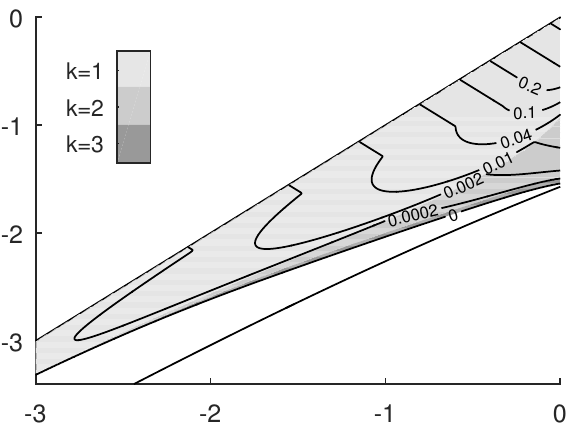}
\put(-165,100){{\footnotesize $\sigma$}}
\put(-28,0){{\footnotesize $\mu$}}
\end{center}
\caption{Lower bounds on $\delta$, with $a=c=1$ for parameter values $(\mu,\sigma)$ in parts of the wedge $\wedg$ so that the ball
$\{\phi:\|\phi\|<\delta\}$ is contained in the basin of attraction of the steady state. The SDDE is initially integrated through $k\tau$ time units so that bounds can be established on derivatives of the solution. The value of $k\in\{1,2,3\}$ which results in the largest bound is indicated.} \label{fig:delta12}
\end{figure}

\begin{defn}[Stability region $\Sigma_\star$]
\label{Defn:StabilityRegion}
Let $a>0$. Let $\Sigma_\star$ be the open set of the $(\mu,\sigma)$-parameter space between the curves
\[ \ell_\star = \bigl\{(s,-s) : s\in(-\infty,1/a]\bigr\},
\quad g_\star= \bigl\{(\mu(s),\sigma(s)) : s\in(0,\pi/a)\bigr\} \]
where the functions $\mu(s)$ and $\sigma(s)$ are given by
\begin{equation}  \label{Eq:musigma-bounds}
  \mu(s)  =  s\cot(as), \quad  \sigma(s)  = - s\csc(as).
\end{equation}
The stability region $\Sigma_\star$ is further divided into three subregions: the cone $\cone=\left\{ (\mu,\sigma) : |\sigma|<-\mu\right\}$, the wedge $\wedg=(\Sigma_\star \setminus \cone )\cap \left\{\mu< 0\right\}$ and the cusp $\cusp=\Sigma_\star\cap \left\{\mu\geq0\right\}$.
\end{defn}

The cone $\cone$ is called the delay-independent stability region because this does not change when the value of $a$ is changed. The remaining region $\wedg\cup\cusp$ is referred to as the delay-dependent stability region. These regions are illustrated in Figure~\ref{Fig:ConstantDelay_divisions}.

In \cite{MH:2016} Lyapunov-Razumikhin techniques were applied to the SDDE \eqref{Eqn:1Delay} to show asymptotic stability in parts of the wedge $\wedg$ and the cusp $\cusp$. Those results establish a lower bound on the radius $\delta$ of the largest ball contained in the basin of attraction of the steady state, specifically it is shown that for initial functions $\phi$ with $\|\phi\|<\delta$ where $\|\phi\|=\sup_{\theta\in[-r,0]}|\phi(\theta)|$ that $\lim_{t\to\infty}u(t)=0$. The resulting bounds within the cusp $\cusp$ are shown in Figure 9 in \cite{MH:2016}. Here, in Figure 2, we show the resulting bounds for $\delta$ within the wedge $\wedg$. While $\delta$ is close to $a/c$ near to $(\mu,\sigma)=(0,0)$ (the $\delta=0.75$ contour can be seen), away from this point $\delta$ decreases rapidly to be close to $0$ in much of the upper part of $\wedg$, while near the lower boundary of $\wedg$ it is not possible to establish asymptotic stability (or a $\delta>0$) with the Lyapunov-Razumikhin method of \cite{MH:2016}.

\section{Global asymptotic stability and bounds on periodic orbits}
\label{Secn:DerivingBounds}
In this section we formulate our generalised Lyapunov-Razumikhin technique and apply it to
the model SDDE~\eqref{Eqn:1Delay} with $a,c>0$ when $\mu<0$ so that Theorem~\ref{Thm:1DelayProperties} ensures that solutions remain bounded within the interval $(M,N)$.

In \cite{MH:2016} we proved Lyapunov-Razumikhin stability results for SDDEs by considering solutions about to exit some ball and deriving a contradiction. Those results relied on three main ideas. Firstly,
we used a Lipschitz condition on $f$ to bound the solution up to some finite time $t^*$.
Secondly, for $t>t^*$ we replace the DDE by a non-autonomous auxiliary ordinary differential equation (ODE) where the delayed term $u(\alpha(t,u(t))$ is replaced by a source term $\eta(t)$ in the ODE formulation.
The source term is required to satisfy constraints that can be derived from properties of $u(\alpha(t,u(t)))$, which in turn follow from properties of the solution $u(t)$ up to the given time $t^*$. This leads to a Lyapunov stability result for SDDEs similar in spirit to the constant delay result of Barnea~\cite{Barnea:1}. Thirdly, we established local asymptotic stability by showing that the existence of a trajectory not asymptotic to the steady-state would result in a contradiction. This last result is quite different to previous Lyapunov-Razumikhin asymptotic stability results, such as those described in Chapter 5 of \cite{HaleLunel:1}, which typically show uniform asymptotic stability, but which require auxiliary functions to enforce the contraction.


In the current work we will consider scalar SDDEs with bounded solutions for which the
solution must either ultimately converge monotonically to an equilibrium, or go through an infinite sequence of alternating relative maxima and minima. Instead of considering solutions about to escape some ball, we generalise
our previous Lyapunov-Razumikhin methods to consider the solutions at their local extrema. We will
replace the DDE by an ODE over an interval $[t^*-\tau,t^*]$ where
$u(t^*)$ is a local extremum of the solution, and $\tau$ is an upper bound on the delay. This will enable us to derive a sequence of bounds on possible relative maxima and minima of the solution. The limit of these bounds yield the bounds on persistent dynamics of the system. If the limit of the upper bounds coincide with the limit of the lower bonds, this yields the global asymptotic stability of the equilibrium.


In Theorem~\ref{Thm:cone} we show that the steady state of \eqref{Eqn:1Delay}
is globally asymptotically stable in the cone $\cone$, the delay-independent stability region.
In Theorem~\ref{Thm:1Delay_changingbounds} we prove recursive bounds on solutions which can be used to bound all the recurrent dynamics, and in
Theorem~\ref{Thm:GAS} establish a condition for global asymptotic stability that can be numerically verified.

We first establish global asymptotic stability of the steady state of \eqref{Eqn:1Delay} for $(\mu,\sigma)\in\cone$, the delay-independent stability region. This is straightforward, and can be proved several different ways.
Recall \eqref{Eqn:1Delay:L0M0tau0} where $M_0$, $N_0$ $\tau_0$ and $\tau$ were defined.

\begin{thm}[Global Asymptotic stability for \eqref{Eqn:1Delay} in $\cone$]  \label{Thm:cone}
Let $a,c>0$, $\mu<-|\sigma|$, so $(\mu,\sigma)\in\cone$.
If $\sigma>0$ let $N\geq0$ and suppose $\varphi(t)\in[M_0,N]$ and continuous for $t\in[-\tau,0]$.
If $\sigma\leq0$ let $N=\max\{\phi(0),N_0\}$ and suppose $\varphi(t)\geq M_0$ and continuous for $t\in[-\tau,0]$.
Then any solution $u(t)$ to \eqref{Eqn:1Delay} satisfies
$u(t)\in(M_0,N)$ for all $t>0$ and
$u(t)\to 0$ as $t\to\infty$.
\end{thm}

\begin{proof}
Theorem~\ref{Thm:1DelayProperties} ensures that $u(t)\in (M_0,N)$ for all $t>0$. \thmlocalcone shows that
$u(t)\to0$ as $t\to\infty$ provided $|u(t)|\leq|M_0|$ for all $t$ sufficiently large.
Since $|\sigma|<|\mu|$ we have $|N_0|<|M_0|$. Thus if $N\leq |M_0|$ the
result follows from \thmlocalcone. So suppose $N>|M_0|$.

If $\sigma\leq0$ and $u(t)\geq|M_0|$ then
$\dot{u}(t)\leq(\mu+\sigma)|M_0|<0$ so for all $t$ sufficiently large $u(t)<|M_0|$ and the result follows.

If $\sigma>0$ then $u(t)\geq\tfrac12(1-\sigma/\mu)N$ implies $\dot{u}(t)\leq\tfrac12(\mu+\sigma)N<0$, and hence
there exists $T_1>0$ such that $u(t)\leq\tfrac12(1-\sigma/\mu)N$ for all $t\geq T_1$,
where $\mu<-|\sigma|$ implies $\tfrac12(1-\sigma/\mu)\in[1/2,1)$.
Then
$u(t-a-cu(t))$ satisfies the same bound for $t\geq T+\tau$, and repeating the same argument
we find a sequence $\{T_k\}$ 
such that $u(t)\leq\tfrac{1}{2^k}(1-\sigma/\mu)^kN$ for $t\geq T_k$. Hence for all
$t$ sufficiently large $u(t)\leq |M_0|$ and the result follows.
\end{proof}


Since Theorem~\ref{Thm:1DelayProperties} ensures that
solutions to \eqref{Eqn:1Delay} remain bounded, we can show that there are only two possible
solution behaviours when $\mu$ and $\sigma$ are both negative.

\begin{lem}\label{Lem:1Delay_solutionbehaviour}
Let $a,c>0$, $\mu<0$ and $\sigma<0$.
Let $\varphi(t)\geq M_0$
be continuous for $t\in[-\tau,0]$ where $\tau=a+cN$ and $N=\max\{\phi(0),N_0\}$.
Then every solution $u(t)$ to \eqref{Eqn:1Delay} must be behave in one of the following manners:
\begin{enumerate}[(A)]
\item There exists $\bar T$ such that $u(t) \to  0$ monotonically for $t>\bar T$
\item For every $T>0$ there exists $T_1, T_2>T$ such that the solution attains a positive local maximum at $T_1$ and a negative local minimum at $T_2$. Furthermore, there exists $t\in[T,T+\tau]$ such that $u(t)\dot u(t)<0$.
\end{enumerate}
\end{lem}

\begin{proof}
Theorem~\ref{Thm:1DelayProperties} ensures that $u(t)\in (M_0,N)$ for all $t>0$
and $\alpha(t,u(t))=t-a-cu(t)\in(t-\tau,t)$.
Suppose there exists a time $T$ such that $u(t)$ does not change sign for $t\geq T$. Suppose first
$u(t)\geqslant 0$ for all $t>T$. Then for $t>\bar T= T+\tau$, we have $\dot u(t)\leqslant 0$ because of \eqref{Eqn:alphabounds}. Hence there is some $\bar u \in[0,N)$ such that
$u(t) \downarrow \bar u$ when $t>\bar T$. Again because of \eqref{Eqn:alphabounds} we must
also have $u(t-a-cu(t))\to \bar u$. This implies $\dot u(t)\to(\mu +\sigma)\bar u$ which is
only possible if $\bar u=0$. Thus in this case, we have the behaviour in (A). The case $u(t)\leq0$
for all $t>T$ is similar.

If there is no time $T$ such either $u(t)\geqslant 0$ or $u(t)\leqslant 0$ for all $t\geqslant T$ then the solution must always be changing signs which yields the behaviour given by the first part of (B).
The second part is proved by observing that if
$u(t)\dot u(t)\geq0$
for all $t\in[T-\tau,T]$
then $\dot{u}(T)$, $u(T)$ and $u(T-a-cu(T))$ would all have the same sign, which
yields an immediate contradiction from \eqref{Eqn:1Delay}.
\end{proof}


Theorem~\ref{Thm:1DelayProperties} and Lemma~\ref{Lem:1Delay_solutionbehaviour} ensure that for $a>0$ , $c>0$, $\mu<0$, $\sigma<0$, solutions to \eqref{Eqn:1Delay} are bounded and either converge to the steady-state or oscillate forever.
Since Theorem~\ref{Thm:cone} already deals with the case of $(\mu,\sigma)\in\cone$,
in the rest of this section we consider $\sigma\leq\mu<0$, which includes the wedge $\wedg$.
The steps that we will take to find bounds on solutions are related to the Lyapunov-Razumikhin techniques we applied in~\cite{MH:2016}, except
instead of considering solutions which are about to exit some ball,
using Lemma~\ref{Lem:1Delay_solutionbehaviour}(B) we consider solutions that achieve a local extremum.

Integrating \eqref{Eqn:1Delay} from $t_0$ to $t$ yields,
\begin{equation}u(t) = u(t_0)e^{ \mu (t-t_0)}  + \sigma \int_{- (t-t_0)}^{0} {e^{-\mu \theta } u\left( \theta+t-a -cu(\theta+t)\right)d\theta}. \label{Eqn:BoundsIntegration0}
\end{equation}
Let $k=1$ or $2$. Suppose that for some
$M\in[M_0,0)$ and $N>0$, $u(s)\in[M,N]$ for $s\in[t-(k+1)\tau,t]$. Suppose also that $u(t)=v$,
a local extremum of the solution. Then $\dot u(t)=0$ and from \eqref{Eqn:1Delay}
\begin{equation}  \label{Eqn:Bounds:EtaConditionA}
u(t-a-cv) = -\frac{\mu}{\sigma} v.
\end{equation}
Let $t_0=t-a-cv$. Then \eqref{Eqn:BoundsIntegration0} becomes
\begin{equation} \label{Eqn:BoundsIntegration1}
v = -\frac{\mu}{\sigma} v e^{\mu \left(a+cv\right)}
+ \sigma \int_{- \left(a+cv\right)}^{0} {e^{-\mu \theta } u\left( \theta+t-a -cu(\theta+t)\right)d\theta}.
\end{equation}
Since $\theta +t-a-cu(t+\theta)\geqslant t-(k+1)\tau$ for all $\theta \in[-a-cv,0]$ for $k\geq1$,
it follows that $u(\theta +t-a-cu(t+\theta))\in[M,N]$ for all $\theta \in[-a-cv,0]$. If $k\geq2$
from the bounds on the solution we also obtain
\begin{align} \notag
& \left|\tfrac{d}{d\theta} u\bigl(\theta+t-a -cu(\theta+t)\bigr)\right|
 = \bigl|1-c\udot(\theta+t)\bigr|\;\bigl|\udot\bigl(\theta+t-a -cu(\theta+t)\bigr)\bigr| \\
& \qquad\quad \leqslant (|\mu|+|\sigma|)\bigl[ 1 + (|\mu|+|\sigma|)\max(|M|,|N|)c \bigr]\max(|M|,|N|),
\text{ for }\theta\in[-a-cv,0].  \label{Eqn:Bounds:EtaConditionB}
\end{align}
Note also that when $u(t)=v$ and $\theta=0$
$$u\left( \theta+t-a -cu(\theta+t)\right) = u(t-a-cu(t))=-\frac{\mu}{\sigma} v.$$
This leads us to define functions $H_{(k)}(v,x,y)(\theta)$ to bound the integrand terms in \eqref{Eqn:BoundsIntegration1}.
For $a>0$, $c>0$ and $\sigma\leqslant\mu<0$, for $\theta\in[-\tau,0]$ we let 
\begin{gather} \label{Eqn:H1vxy}
H_{(1)}(v,x,y)(\theta) = \left\{ \begin{array}{ll}
\phantom{-} 
x, & \theta <0,  \\
-\frac{\mu}{\sigma}v, & \theta=0,
 \end{array}  \right.  \\  \label{Eqn:H2vxy}
H_{(2)}(v,x,y)(\theta)
= \left\{ \begin{array}{ll}
\phantom{-}  
x, &
\theta  \leqslant \text{sign}(y-x)\frac{x +\mu v/\sigma}{D(x,y)},  \\
 - \frac{\mu }{\sigma }v +\text{sign}(y-x)D(x,y)\theta ,
& \theta  > \text{sign}(y-x)\frac{x +\mu v/\sigma}{D(x,y)},  \\
 \end{array}  \right.
\end{gather}
where we take $\text{sign}(0)=0$ and
\begin{equation} \label{Eqn:Dxy}
D(x,y) = (|\mu|+|\sigma|)\bigl[ 1+(|\mu|+|\sigma|)\max(|x|,|y|)c\bigr]\max(|x|,|y|).
\end{equation}
The functions $H_{(k)}(v,M,N)(\theta)$ and $H_{(k)}(v,N,M)(\theta)$ are respectively the most negative and positive possible forms of $u\bigl(\theta+t-a-cu(\theta+t)\bigr)$
which satisfy \eqref{Eqn:Bounds:EtaConditionA}, and
when $k=2$ also satisfy the derivative bound \eqref{Eqn:Bounds:EtaConditionB}. 

We will use the functions $H_{(k)}(v,x,y)(\theta)$ below to bound the the possible dynamics of solutions. We formalised the use of source terms with suitable constraints to replace the delayed terms for SDDEs in \cite{MH:2016}, 
where the functions $\eta_{(k)}(\theta)$ correspond to $H_{(k)}(\delta,-\delta,\delta)(\theta)$. But, the use of source terms to replace delay terms has a very long history in DDE theory. For Lyapunov-Razumikhin results, Barnea~\cite{Barnea:1} already used similar functions and bounds when studying stability of Hayes equation, the constant delay version of \eqref{Eqn:1Delay} with $c=0$.


Suppose now that $u(t)=v$ is a negative relative minimum, then
\begin{equation} \label{Eqn:relmin}
v \geqslant  - \frac{\mu}{\sigma}ve^{\mu(a + cv)}
+ \sigma \int_{ - a - cv}^0 e^{-\mu\theta} H_{(k)}(v,N,M)(\theta)d\theta.
\end{equation}
If instead $u(t)=v$ is a positive relative maximum then we obtain
\begin{equation} \label{Eqn:relmax}
v \leqslant  - \frac{\mu}{\sigma}ve^{\mu(a + cv)}
+ \sigma \int_{ - a - cv}^0 e^{-\mu\theta} H_{(k)}(v,M,N)(\theta)d\theta .
\end{equation}
Now let
\begin{equation} \label{Eqn:Hkvxy}
\cQ_{(k)}(v,x,y) = v + \frac{\mu}{\sigma}ve^{\mu(a + cv)}
- \sigma \int_{ - a - cv}^0 {e^{-\mu\theta}H_{(k)}(v,x,y)(\theta)d\theta }.
\end{equation}
Then \eqref{Eqn:relmin} and \eqref{Eqn:relmax} can be written as $\cQ_{(k)}(v,N,M)\geqslant 0$ and $\cQ_{(k)}(v,M,N)\leqslant 0$ respectively.



It follows from \eqref{Eqn:H1vxy} and \eqref{Eqn:H2vxy} that
$H_{(1)}(v,x,y)(\theta)\leq H_{(2)}(v,x,y)(\theta)$ if $x\leq-\mu v/\sigma$. Hence if $v\leq-\sigma x/\mu$ we have
\begin{equation} \label{Eqn:H2>H1}
\cQ_{(2)}(v,x,y)\geqslant \cQ_{(1)}(v,x,y).
\end{equation}
Similarly, if $v\geq-\sigma x/\mu$ then
\begin{equation} \label{Eqn:H2<H1}
\cQ_{(2)}(v,x,y)\leqslant \cQ_{(1)}(v,x,y).
\end{equation}
Substituting \eqref{Eqn:H1vxy} into \eqref{Eqn:Hkvxy} we obtain
\begin{equation}  \label{Eqn:H1}
\cQ_{(1)}(v,x,y) =\left(1+\frac{\mu}{\sigma}  e^{\mu (a+cv)} \right)v+\frac{\sigma }
{\mu }x\left(1-e^{\mu(a+cv)}\right).
\end{equation}
Substituting \eqref{Eqn:H2vxy} into \eqref{Eqn:Hkvxy} there are two cases to consider.
If $-\text{sign}(y-x)\frac{{x +\tfrac{\mu }{\sigma }v}}{{D(x,y)}}< a+cv$
\begin{equation} \label{Eqn:H2_2part}
\cQ_{(2)}(v,x,y)
=\frac{\mu}{\sigma} v e^{\mu(a+cv)}
-\frac{\sigma}{\mu}\biggl[ {\text{sign}(y-x)\frac{D(x,y)}{\mu}
\Bigl( {e^{-\mu\text{sign}(y-x)\tfrac{{x +\tfrac{\mu}{\sigma}v}}{{D(x,y)}}}  - 1} \Bigr) +x e^{\mu (a+cv)} }\biggr],
\end{equation}
while if $-\text{sign}(y-x)\frac{{x +\tfrac{\mu}{\sigma}v}}{{D(x,y)}}\geqslant a+cv$
\begin{equation} \label{Eqn:H2_1part}
\cQ_{(2)}(v,x,y)
 = v e^{\mu(a+cv)}\left(1+\frac{\mu}{\sigma}\right)
 - \text{sign}(y-x)\frac{\sigma}{\mu}D(x,y)
 \left[\frac{1}{\mu}\left( {e^{\mu (a+cv)}  - 1} \right) - \left(a+cv\right) e^{\mu (a+cv) } \right].
\end{equation}

A version of the function $\cQ_{(1)}(v,x,y)$ was already introduced in the more general case of multiple state-dependent delays in Section 3 of \cite{Humphries:1}.
We will use $\cQ_{(k)}(v,x,y) $ to derive bounds on the solutions of \eqref{Eqn:1Delay}. Given an
initial lower bound $M$ and upper bound $N$ for the solution, we can derive a new lower
bound $\mathcal{M}_{(k)}(N,M)$ and upper bound $\mathcal{N}_{(k)}(M,N)$
which are valid after some time,
where $\mathcal{M}_{(k)}$ and
$\mathcal{N}_{(k)}$ are functions defined implicitly using $\cQ_{(k)}(v,x,y)$ in Lemmas~\ref{Lem:H1}--\ref{Lem:H2}.
The process by which the bounds can be applied
is described in Theorem~\ref{Thm:1Delay_changingbounds}. If these shrinking bounds converge to the same limit as
$t\to\infty$, global asymptotic stability of the steady state
is obtained, as described in Theorem~\ref{Thm:GAS}.
To prove these theorems, we first need to consider the derivatives of $\cQ_{(k)}(v,x,y)$ in
Lemma~\ref{Lem:Hderivatives} and derive the $\mathcal{M}_{(k)}$ and $\mathcal{N}_{(k)}$ functions
in Lemma~\ref{Lem:H1} (for $k=1$) and Lemma~\ref{Lem:H2} (for $k=2$).

\begin{lem} \label{Lem:Hderivatives}
Let $a>0$, $c>0$ and $\sigma\leqslant \mu<0$.
The functions $\cQ_{(1)}(v,x,y)$, $\cQ_{(2)}(v,x,y)$ and the derivatives
\begin{align} \label{Eqn:ddvH1}
&\frac{\partial }{{\partial v}}\cQ_{(1)}(v,x,y) = 1 + \frac{\mu }
{\sigma }\left( 1 + \mu c v - \tfrac{\sigma^2 c}{\mu}x \right)e^{ \mu \left( {a + cv} \right)},
\\ \notag
&\frac{\partial }
{{\partial v}}\cQ_{(2)} (v,x,y) \\ \notag
&\; =\left\{
\begin{array}{ll}
\rule[-8pt]{0pt}{10pt}\frac{\mu}{\sigma}\left( 1 + \mu cv - \tfrac{\sigma ^2 cx}{\mu } \right)
e^{\mu(a + cv)}  + e^{- \mu\; \mathrm{sign} (y-x)\tfrac{{x + \frac{\mu}{\sigma}v}}{D(x,y)}},
& -\mathrm{sign} (y-x)\tfrac{{x + \tfrac{\mu}{\sigma}v}}
{{D(x,y)}}< a+cv, \\
\left[ ( 1 + \mu cv )
\left( 1 + \frac{\mu}{\sigma} \right) + \mathrm{sign}(y-x)\sigma cD(x,y)(a + cv) \right]
e^{\mu(a + cv)},
 &-\mathrm{sign}(y-x)\tfrac{{x + \tfrac{\mu}{\sigma}v}}
{{D(x,y)}}\geqslant a+cv.
\end{array}
\right.
\end{align}
are continuous in the subsets of $\mathbb{R}^3$ where either of the following cases hold:
\begin{enumerate}[i)]
\item $v\in[y,0]$, $x\in[0,N_0]$ and $y\in[M_0,0]$,
\item $v\in[0,y]$, $x\in[M_0,0]$ and $y\in[0,N_0]$.
\end{enumerate}
In these subsets, we also note that $-\mathrm{sign}(y-x)\frac{x+\frac{\mu}{\sigma}v}{D(x,y)}\geq0$.
\end{lem}

\begin{proof}
It is easy to derive the expressions for the derivatives from \eqref{Eqn:H1}--\eqref{Eqn:H2_1part}.
Continuity of $\frac{\partial}{\partial v} \cQ_{(2)}(v,x,y)$ is shown by setting
$-\text{sign}(y-x)\tfrac{x +\mu v/\sigma}{D(x,y)}=a+cv$
and comparing the two expressions.
\end{proof}

\begin{lem}
\label{Lem:H1}
Let $a>0$, $c>0$ and $\sigma\leqslant \mu<0$.
\begin{enumerate}
\item
For every fixed $x\in[0,N_0]$ and $y<0$ there is an
$\mathcal{M}_{(1)}(x,y)\in \bigl[\max\{M_0,-\sigma x/\mu\},0\bigr]$
such that $\cQ_{(1)}(v,x,y)$ is negative if $v<\mathcal{M}_{(1)}(x,y)$, zero if
$v=\mathcal{M}_{(1)}(x,y)$ and positive if $v\in\left(\mathcal{M}_{(1)}(x,y),0\right]$.
\item
$\mathcal{M}_{(1)}(x,y)\in\cC\bigl([0,N_0]\times(-\infty,0],[M_0,0)\bigr)$ is a decreasing function of $x$.
\item
For every fixed $x\in[M_0,0]$ and $y>0$ there is an $\mathcal{N}_{(1)}(x,y)\in[0,-\sigma x/\mu]$
such that $\cQ_{(1)}(v,x,y)$ is negative if $v\in\left[0,\mathcal{N}_{(1)}(x,y)\right)$, zero if $v=\mathcal{N}_{(1)}(x,y)$, and positive if $v>\mathcal{N}_{(1)}(x,y)$.
\item
$\mathcal{N}_{(1)}(x,y)\in\cC\bigl([M_0,0]\times[0,\infty),(0,N_0]\bigr)$ is a decreasing function of $x$.
\end{enumerate}
\end{lem}

\begin{proof}
If $x=0$ then $v=0$ is the only solution to $\cQ_{(1)}(v,x,y)=0$, so we consider $x\ne 0$.
Let $x\in(0,N_0]$ and $y<0$. Then $H_{(1)}(M_0,x,y)<0$, $\cQ_{(1)}\bigl(-\sigma x/\mu,x,y\bigr)<0$,
$\cQ_{(1)}(0,x,y)>0$ and $\tfrac{\partial}{\partial v}\cQ_{(1)}(v,x,y)> 0$ for all $v<0$. Thus $\cQ_{(1)}(v,x,y)$ is monotonically increasing for all $v<0$ and changes sign in the interval $\bigl[\max\{M_0,-\sigma x/\mu\},0\bigr]$.
This leads to Property 1. This property and the continuity of $\cQ_{(1)}(v,x,y)$ in $x$ and $y$ allows $\mathcal{M}_{(1)}(x,y)$
to be defined as a continuous function.

Let $\theta\in[-a-cv,0]$. As $x\in[0,N_0]$ increases, $H_{(1)}(v,x,y)(\theta)=x$ increases for $\theta\in[-a-cv,0]$ so $\cQ_{(1)}(v,x,y)$ also increases, which implies that $\mathcal{M}_{(1)}(x,y)$ decreases. This completes the proof of Property 2.

Now let $x\in[M_0,0)$ and $y>0$. Then $\cQ_{(1)}(0,x,y)=\frac{\sigma}{\mu}x(1-e^{\mu a})<0$
and $\cQ_{(1)}(-\sigma x/\mu,x,y)=-(1+\sigma/\mu)xe^{\mu(a-c\sigma x/\mu)}>0$. Thus there is a solution
to $\cQ_{(1)}(v,x,y)=0$ with $v\in(0,-\sigma x/\mu)$. To show that this solution is unique, we
differentiate $\cQ_{(1)}$. From \eqref{Eqn:ddvH1}
any solution $v=v_*$ to $\tfrac{\partial}{\partial v}\cQ_{(1)}(v,x,y)=0$ is given by
$$\left( 1 + \mu c v_* - \tfrac{\sigma^2c}{\mu}x \right)
e^{1 + \mu cv_* - \tfrac{\sigma^2c}{\mu}x} = -\tfrac{\sigma}{\mu}e^{1 -\mu a -\tfrac{c\sigma^2}{\mu}x} $$
and hence
\begin{equation} \label{Eqn:vstar1}
v_*=
\frac{1}{c\mu}(W(z)-1) + \frac{\sigma^2}{\mu^2}x, \qquad
z= -\frac{\sigma}{\mu}e^{1 -\mu a -\tfrac{c\sigma^2}{\mu}x}.
\end{equation}
where $W(z)$ is the Lambert W-function \cite{Corless1996}. Recall that the real-valued Lambert W-function
$W\in C\bigl([-e^{-1},\infty),\mathbb{R}\bigr)$ is double-valued on $(-e^{-1},0)$.
From \eqref{Eqn:ddvH1} and \eqref{Eqn:vstar1}
if $\tfrac{\partial}{\partial v}\cQ_{(1)}(v_*,x,y)=0$ then
$1+\tfrac{\mu}{\sigma}W(z)e^{\mu(a+cv_*)}=0$.
But $\mu/\sigma\in(0,1]$ and $e^{\mu(a+cv_*)}\leqslant 1$, so for such a $v_*$ we require
$W(z)\leq-1$, and in \eqref{Eqn:vstar1} we only need to consider the restricted
Lambert W-function $W\in C\bigl([-e^{-1},0],(-\infty,-1]\bigr)$, which is injective.
Thus there can be at most one point for which
$\tfrac{\partial}{\partial v}\cQ_{(1)}(v,x,y)$ changes sign
(and if $z\geqslant-e^{-1}$ such a point $v_*$ exists).
Since the derivative can change sign at most once,
there is exactly one solution $v=\mathcal{N}_{(1)}(x,y)$ to $\cQ_{(1)}(v,x,y)=0$ for $v\in(0,-\sigma x/\mu)$.
This completes the proof of Property 3. This property and the continuity
of $\cQ_{(1)}(v,x,y)$ in $x$ and $y$ allows $\mathcal{N}_{(1)}(x,y)$
to be defined as a continuous function.

Let $\theta\in[-a-cv,0]$. As $x\in[M_0,0]$ increases,
$H_{(1)}(v,x,y)(\theta)=x$ increases. So $\cQ_{(1)}(v,x,y)$ is an increasing function of $x$, and hence
$\mathcal{N}_{(1)}(x,y)$ is a decreasing function of $x$. This completes the proof of Property 4.
\end{proof}

\begin{lem} \label{Lem:H2}
Let $a>0$, $c>0$, $\sigma\leqslant \mu<0$.
\begin{enumerate}
\item
For every fixed $x\in[0,N_0]$ and $y\in[M_0,0)$ there is an
$\mathcal{M}_{(2)}(x,y)\in\bigl[\mathcal{M}_{(1)}(x,y),0\bigr]$ such that $\cQ_{(2)}(v,x,y)$ is negative if $v<\mathcal{M}_{(2)}(x,y)$, zero if $v=\mathcal{M}_{(2)}(x,y)$, and positive if
$v\in\bigl(\mathcal{M}_{(2)}(x,y),0\bigr]$.
\item
$\mathcal{M}_{(2)}(x,y)\in\cC\bigl([0,N_0]\times(-\infty,0],[M_0,0)\bigr)$ is a decreasing function of $x$.
\item
For every fixed $x\in[M_0,0]$ and $y\in(0,N_0]$, there is an $\mathcal{N}_{(2)}(x,y)\in \bigl[0,\mathcal{N}_{(1)}(x,y)\bigr]$
such that $\cQ_{(2)}(v,x,y)$ is negative if $v\in\bigl[0,\mathcal{N}_{(2)}(x,y)\bigr)$,
zero if $v=\mathcal{N}_{(2)}(x,y)$, and positive if $v>\mathcal{N}_{(2)}(x,y)$.
\item
$\mathcal{N}_{(2)}(x,y)\in\cC\bigl([M_0,0]\times[0,\infty),(0,N_0]\bigr)$ is a decreasing function of $x$.
\end{enumerate}
\end{lem}

\begin{proof}
By the results for $k=1$ and the inequalities \eqref{Eqn:H2>H1} and \eqref{Eqn:H2<H1}, the case $x=0$ will always yield a solution $v=0$. So we consider $x\ne 0$.

Let $x\in(0,N_0]$ and $y\in[M_0,0)$. By Lemma~\ref{Lem:H1}(1) we have
$\mathcal{M}_{(1)}(x,y)\geq-\sigma x/\mu$ and hence by
\eqref{Eqn:H2<H1}, $\cQ_{(2)}\bigl(\mathcal{M}_{(1)}(x,y),x,y\bigr)\leqslant
\cQ_{(1)}\bigl(\mathcal{M}_{(1)}(x,y),x,y\bigr)=0$.
By \eqref{Eqn:H2_2part}--\eqref{Eqn:H2_1part}, $\cQ_{(2)}(0,x,y) > 0$.
Thus there must exist a solution $v\in[\mathcal{M}_{(1)}(x,y),0)$ to $\cQ_{(2)}(v,x,y)=0$.
Using the expressions from Lemma~\ref{Lem:Hderivatives} it follows that $\tfrac{\partial}{\partial v}\cQ_{(2)}(v,x,y)>0$ for
$v\in(M_0,0)$.
Thus the solution is unique and Property 1 follows easily. This property and the continuity of $\cQ_{(2)}(v,x,y)$
in $x$ and $y$ allows $\mathcal{M}_{(2)}(x,y)$ to be defined as a continuous function.

For fixed $v$ and $y\in(M_0,0)$, if $x\in(0,N_0]$ increases, $H_{(2)}(v,x,y)(\theta)$ increases for all $\theta\in[-a-cv,0)$ so $\cQ_{(2)}(v,x,y)$ also increases. By Property 1, this implies
$\mathcal{M}_{(2)}(x,y)$ decreases with increasing $x$. This completes the proof of Property 2.

Now let $x\in[M_0,0)$ and $y>0$. Then $\cQ_{(2)}(0,x,y)<0$.
We note that when $v=-\sigma x/\mu$ we have $0=-\text{sign}(y-x)\tfrac{x + (\mu/\sigma) v}{D(x,y)}< a+cv$.
At this point
$H_{(1)}(v,x,y)(\theta)=H_{(2)}(v,x,y)(\theta)=x$ for all $\theta\in[-a-cv,0]$
and
$\cQ_{(2)}(-\sigma x/\mu,x,y) = \cQ_{(1)}(-\sigma x/\mu,x,y)>0$ (from Lemma~\ref{Lem:H1}).
Thus there must exist
$v\in(0,-\sigma x/\mu]\subseteq(0,N_0]$ which solves
$\cQ_{(2)}(v,x,y)=0$.

For the behaviour of $\tfrac{\partial}{{\partial v}}\cQ_{(2)}(v,x,y)$
over $v\in(0,-\sigma x/\mu)$, there are two cases to consider:
\begin{enumerate}
\item $-\text{sign}(y-x)\tfrac{x +\mu v/\sigma}{D(x,y)}< a+cv$ for all $v\in(0,-\sigma x/\mu)$. Then, similar
to the proof of Lemma~\ref{Lem:H1},
using the restricted Lambert W-function, we find that
$\tfrac{\partial}{\partial v}\cQ_{(2)} (v,x,y)$ can change sign at most once at $v_*$, where
\[
v_*=\tfrac{1}{c\mu}(W(z)-1) + \tfrac{\sigma^2}{\mu^2}x, \qquad
z= - \tfrac{\sigma}{\mu}e^{1 -\mu a +\mu\text{sign}(x)\tfrac{x+\mu v/\sigma}{D(x,y)}-\tfrac{c\sigma^2}{\mu}x}.
\]
Since $\cQ_{(2)}(0,x,y)<0<\cQ_{(2)}(\sigma x/\mu,x,y)$, whether or not
$v_*\in(0,-\sigma x/\mu)$ there can still only be one solution $v=\mathcal{N}_{(2)}(x,y)$ to $\cQ_{(2)}(v,x,y)=0$.
\item
$-\text{sign}(y-x)\tfrac{x + \mu v_1/\sigma}{D(x,y)}=a+cv_1$ for some $v_1\in(0,-\sigma x/\mu)$.
In this case, for $v\in (0,v_1)$, the sign of
$\tfrac{\partial}{{\partial v}}\cQ_{(2)}(v,x,y)$ depends on the factor,
\begin{multline*}
( 1 + \mu cv)
\left( 1 + \tfrac{\mu}{\sigma} \right) + \text{sign}(y-x)\sigma c D(x,y)(a + cv) \\
=\left( 1 + \tfrac{\mu}{\sigma}\right)
+ \text{sign}(y-x)\sigma c D(x,y)a
+\left[\mu c\left( 1 + \tfrac{\mu }{\sigma } \right) + \text{sign}(y-x)\sigma c^2 D(x,y)a\right]v.
\end{multline*}
Since $\mu c( 1 + \mu/\sigma)<0$ and $\text{sign}(y-x)\sigma c^2 D(x,y)a<0$ then
either $\tfrac{\partial}{{\partial v}}\cQ_{(2)}(v,x,y)>0$ for all $v\in(0,v_1]$ or
$\tfrac{\partial}{{\partial v}}\cQ_{(2)}(v_1,x,y)\leq0$.
We will show that this second case is not possible.
At $v=v_1$, 
$$
\tfrac{\partial}{{\partial v}}\cQ_{(2)}(v_1,x,y)
 = \bigl[\tfrac{\mu}{\sigma}(1+\mu c v_1 -\tfrac{\sigma^2 c x}{\mu} )+1 \bigr] e^{\mu(a+cv_1)}
 \geq \bigl(1+\tfrac{\mu}{\sigma}\bigr)(1+\sigma a)e^{\mu(a+cv_1)},
$$
where the last inequality follows since $v_1\leq N_0$  and $x\geq M_0$.
Thus it is sufficient to show that
$a\sigma>-1$
to prove that $\tfrac{\partial}{{\partial v}}\cQ_{(2)}(v_1,x,y)>0$.

We have $0<v_1=-\tfrac{x+D(x,y) a}{(\mu/\sigma) +c D(x,y)}$ and hence $x+D(x,y)a<0$.
But from \eqref{Eqn:Dxy},
$$x+D(x,y)a\geq x+(|\mu|+|\sigma|)|x|a=|x|(-a\sigma-(1+a\mu)).$$
Thus $-a\sigma-(1+a\mu)<0$ which implies that $a\sigma>-(1+a\mu)>-1$,
and hence $\tfrac{\partial}{{\partial v}}\cQ_{(2)}(v,x,y)>0$ for all $v\in(0,v_1]$.
From the previous case,
$\tfrac{\partial}{{\partial v}}\cQ_{(2)}(v,x,y)$ can only change sign once for $v\in(v_1,-\sigma x/\mu)$
and hence only once for $v\in(0,-\sigma x/\mu)$ and again there is only one solution $v$ to $\cQ_{(2)}(v,x,y)=0$ on $v\in(0,-\sigma x/\mu)$.
\end{enumerate}
In either case, there is exactly one solution $v\in(0,-\sigma x/\mu)$ to $\cQ_{(2)}(v,x,y)=0$ and we define
$\mathcal{N}_{(2)}(x,y)=v$ where $v$ solves $\cQ_{(2)}(v,x,y)=0$.
By Lemma~\ref{Lem:H1}(3) we have $\mathcal{N}_{(1)}(x,y)\leq-\sigma x/\mu$ and hence by
\eqref{Eqn:H2>H1},
$\cQ_{(2)}\bigl(\mathcal{N}_{(1)}(x,y),x,y\bigr)\geq \cQ_{(1)}\bigl(\mathcal{N}_{(1)}(x,y),x,y\bigr)=0$.
It follows that $\mathcal{N}_{(2)}(x,y)\in\bigl(0,\mathcal{N}_{(1)}(x,y)\bigr)$.

Let $\theta\in[-a-cv,0]$. As $x\in[M_0,0]$ increases, $H_{(2)}(v,x,y)(\theta)$ increases.
So $\cQ_{(2)}(v,x,y)$ is an increasing function of $x$.
This means $\mathcal{N}_{(2)}(x,y)$ is a decreasing function of $x$. This completes the proof of Property 4.
\end{proof}

We are now in the position to prove Theorem~\ref{Thm:1Delay_changingbounds} which for $k=1$ or $2$, gives a sequence of lower bounds $M_{k,n}$ and upper bounds $N_{k,n}$ to the solutions of \eqref{Eqn:1Delay}.


\begin{thm} \label{Thm:1Delay_changingbounds}
Let $k=1$ or $2$, $a>0$, $c>0$ and $\sigma\leqslant \mu<0$.
Let $u(t)$ be a 
solution to \eqref{Eqn:1Delay} where
$\varphi(t)\geq M_0$ and 
continuous for $t\in[-\tau,0]$ with $N=\max\{\phi(0),N_0\}$.
Define the sequence of bounds $\{M_{k,n}\}_{n\geq0}$ and $\{N_{k,n}\}_{n\geq0}$ recursively
with $M_{k,0}=M_0$ and $N_{k,0}=N_0$ and using
\begin{equation*}
M_{k,n+1}=\mathcal{M}_{(k)}(N_{k,n},M_{k,n}), \quad N_{k,n+1}=\mathcal{N}_{(k)}(M_{k,n},N_{k,n}).
\end{equation*}
Then there exists a sequence $\{T_{k,n}\}\to\infty$ such that $u(t)\in[M_{k,n},N_{k,n}]$
for all $t\geqslant T_{k,n}$.
\end{thm}


\begin{proof}
By Theorem~\ref{Thm:1DelayProperties} every solution satisfies $u(t)\in(M_0,N)$ for
all $t>0$. If $u(t)\geq N_0$ then $\dot{u}(t)\leq0$ from \eqref{Eqn:1Delay}. Hence if there exists
$T_0\geq0$ such that $u(T_0)\leq N_0$ then $u(t)\in(M_0,N_0]$ for all $t\geq T_0$. But, if $u(t)>N_0$
for $t\geq \tau$ then $u(t-a-cu(t))>N_0$  using \eqref{Eqn:alphabounds}, and
hence $\dot{u}(t)\leq(\mu+\sigma)N_0<0$, and so there must exist $T_0>0$ such that
$u(t)\in(M_0,N_0]$ for all $t\geq T_0$.

By Lemmas~\ref{Lem:H1} and~\ref{Lem:H2}, the sequences $\{M_{k,n}\}$ and $\{N_{k,n}\}$ defined above satisfy $M_{k,n} \leq M_{k,n+1}\leq 0 \leq N_{k,n+1}\leq N_{k,n}$ for $n\in\mathbb{Z}^+$,
and hence define a sequence of intervals with $[M_{k,n+1},N_{k,n+1}] \subseteq [M_{k,n},N_{k,n}]$ for $n\in\mathbb{Z}^+$.

By Lemma~\ref{Lem:1Delay_solutionbehaviour}, either the solution $u(t)\to0$ eventually monotonically,
or for all $t>0$ there exists $T^+, T^->t$ such that the solution attains a
positive local maximum at $T^+$ and a negative local minimum at $T^-$.
We need only consider this last case, and establish the result by induction.

Suppose that $u(t)\in[M_{k,n},N_{k,n}]$ for all $t\geq T_{k,n}$ for some $n\in\mathbb{Z}^+$ (and
note that we already established this for $n=0$ above).
Let $\tau_{k,n}=a+c N_{k,n}$, then
any positive local maximum at $T^+_{k,n+1}\geq T_{k,n}+(k+1)\tau_{k,n}$ must have $u(T^+_{k,n+1})=v^+$ satisfying $\cQ_{(k)}(v^+,M_{k,n},N_{k,n})\leq0$, and hence $v^+\leq \cN_{(k)}(M_{k,n},N_{k,n})=N_{k,n+1}$.
Any negative local minimum at $T^-_{k,n+1} \geq T_{k,n}+(k+1)\tau_{k,n}$ must have $u(T^-_{k,n+1})=v^-$ satisfying
$\cQ_{(k)}(v^-,N_{k,n},M_{k,n})\geq0$, and hence $v^-\geq \cM_{(k)}(N_{k,n},M_{k,n})=M_{k,n+1}$.
Let $T^+_{k,n+1}$ and $T^-_{k,n+1}$ denote the first local maximum and minimum for $t \geq T_n+k\tau_{k,n}$
and let $T_{n+1}=\max\{T^-_{n+1},T^+_{n+1}\}$ then $u(t)\in[M_{k,n+1},N_{k,n+1}]$ for all $t\geq T_{k,n+1}$,
which completes the induction and the proof.
\end{proof}

In the limit as $n \to \infty$, the interval $\left[M_{k,\infty},N_{k,\infty}\right]$ which bounds the solution as $t \to \infty$ can be calculated using the equations
\begin{equation} \label{Eqn:Bifurcation_1Delayinfinity}
\cQ_{(k)}\bigl(N_{k,\infty},M_{k,\infty},N_{k,\infty}\bigr)=\cQ_{(k)}\bigl(M_{k,\infty},N_{k,\infty},M_{k,\infty}\bigr)=0.
\end{equation}
If $M_{k,\infty}=N_{k,\infty}=0$ the steady state is global asymptotic stability.

\begin{thm}[Global asymptotic stability in $\wedg$] \label{Thm:GAS}
Let $k=1$ or $2$, $a>0$, $c>0$ and $\sigma\leqslant \mu<0$.
Let $u(t)$ be the solution to \eqref{Eqn:1Delay} where
$\varphi(t)\geq M_0$ and 
continuous for $t\in[-\tau,0]$ with $N=\max\{\phi(0),N_0\}$.
Suppose that the only solution $(M,N)\in[M_0,0]\times[0,N_0]$
to $\cQ_{(k)}(N,M,N)=\cQ_{(k)}(M,N,M)=0$ is $M=N=0$
then any solution $u(t)$ to \eqref{Eqn:1Delay} satisfies $u(t)\to 0$.
\end{thm}

\begin{proof}
From the proof of Theorem~\ref{Thm:1Delay_changingbounds} $M_0\leq M_{k,\infty}\leq0\leq N_{k,\infty}\leq N_0$.
Then by continuity of $\mathcal{M}_{(k)}$ and $\mathcal{N}_{(k)}$ we have
$$M_{k,\infty}=\mathcal{M}_{(k)}(N_{k,\infty},M_{k,\infty}), \quad N_{k,\infty}=\mathcal{N}_{(k)}(M_{k,\infty},N_{k,\infty}),$$
and hence \eqref{Eqn:Bifurcation_1Delayinfinity} holds. But by supposition
$M_{k,\infty}=N_{k,\infty}=0$ is the only solution of \eqref{Eqn:Bifurcation_1Delayinfinity}.
\end{proof}

\begin{figure}[t]
\begin{center}
\subfigure[]{\includegraphics[scale=0.5]{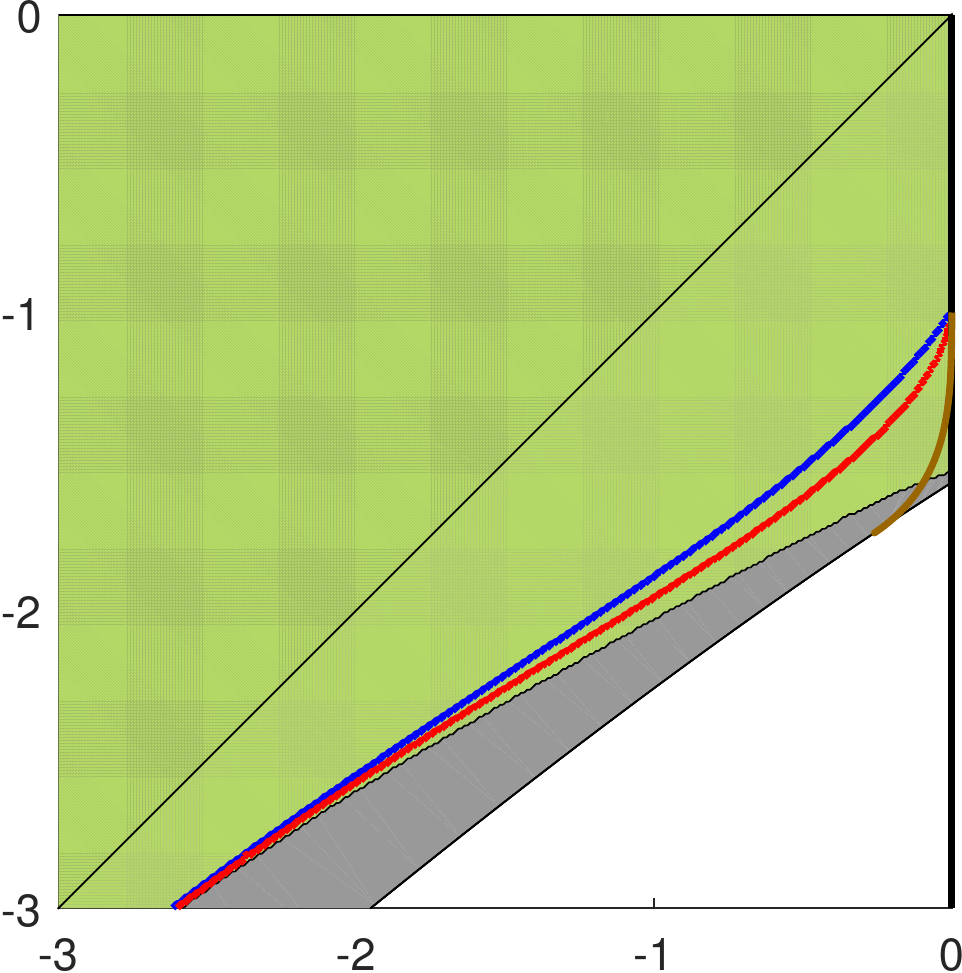}}
\put(-25,0){{\footnotesize $\mu$}}
\put(-145,115){{\footnotesize $\sigma$}}
\hspace{2em}
\subfigure[]{\includegraphics[scale=1]{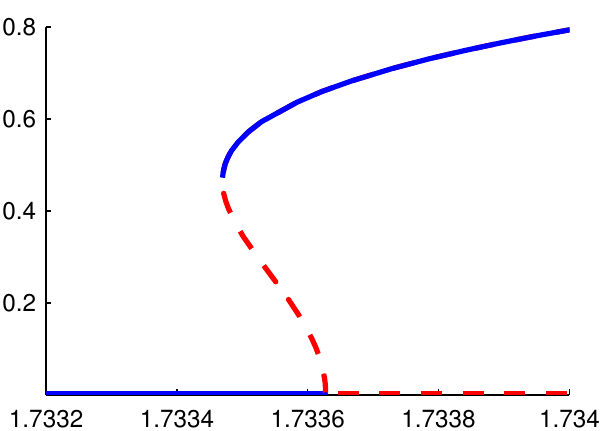}}
\put(-30,-1){{\footnotesize $\sigma$}}
\caption{\label{Fig:GAS}(a) The green region is the parameter region where the zero solution was proven to be locally asymptotically stable using Lyapunov-Razumikhin techniques in  \cite{MH:2016} (denoted as the set $\{P(1,0,3)<1\}$).
The blue curve shows the boundary of the region where we numerically evaluate $M_{1,\infty}=N_{1,\infty}=0$ and the red curve shows the boundary of the set where $M_{2,\infty}=N_{2,\infty}=0$. The steady state of \eqref{Eqn:1Delay} is globally asymptotically stable to the left of the red curve.
The brown curve shows the locus of a fold bifurcation of periodic orbits computed using DDE-Biftool.
To the right of the brown curve there are co-existing non-trivial periodic orbits and the steady-state
cannot be globally asymptotically stable. Other parameters are fixed at $a=c=1$.
(b) Bifurcation diagram at $\mu=-0.25$ as $\sigma$ is varied showing a subcritical Hopf bifurcation from
the steady state and the amplitude of the resulting periodic orbits.
There is a small interval of $\sigma$ values for which bistability occurs, and the steady-state is only locally
asymptotically stable.}
\end{center}
\end{figure}

Theorem~\ref{Thm:GAS} provides a condition for global asymptotic stability that is easy to verify for any fixed set of parameter values. In Figure~\ref{Fig:GAS}(a) we show
the boundary of the regions where $M_{k,\infty}=N_{k,\infty}=0$ for $k=1$ and $k=2$
in the $(\mu,\sigma)$ plane.
By Theorem~\ref{Thm:GAS}, the zero solution to \eqref{Eqn:1Delay} is globally asymptotically stable when $M_{2,\infty}=N_{2,\infty}=0$, to the left of the red curve in Figure~\ref{Fig:GAS}(a).
We note that the steady-state is not globally asymptotically in all of the wedge $\wedg$ because
for $\mu$ sufficiently close to $0$ the Hopf bifurcation which occurs on the boundary of $\wedg$ is
subcritical leading to coexistence of the stable steady state and non-trivial periodic orbits. This
is illustrated in Figure~\ref{Fig:GAS}(b).
The green region is the set 
for which the Lyapunov-Razumikhin results of \cite{MH:2016} gave the lower bounds on the basin of attraction as shown in
Figure~\ref{fig:delta12}.
The brown line in Figure~\ref{Fig:GAS}(a) shows the (numerically computed)
two-parameter locus of the fold bifurcation of periodic orbits seen in Figure~\ref{Fig:GAS}(b).
To the right of this line in Figure~\ref{Fig:GAS}(a) there is coexistence of the steady-state
and non-trivial periodic orbits, and so the steady-state cannot be globally asymptotically stable.
Comparing the red curve, with the brown curve and the green region 
we see that using Theorem~\ref{Thm:GAS} with $k=2$ we are able to establish
global asymptotic stability of the steady state in most of the parameter set in which we showed
that it is locally
asymptotically stable using Lyapunov-Razumikhin techniques in \cite{MH:2016}. The small region of $(\mu,\sigma)$
parameter space where we established local asymptotic stable in \cite{MH:2016}, but are unable to establish global asymptotic stability using the techniques of this paper
includes a region where the steady-state is only locally asymptotically stable due to the subcritical Hopf bifurcation.

\begin{figure}[t]
\begin{center}
\subfigure[$\dot u(t)=-0.25u(t)-1.75u(t-1-u(t))$]{\includegraphics[scale=1]{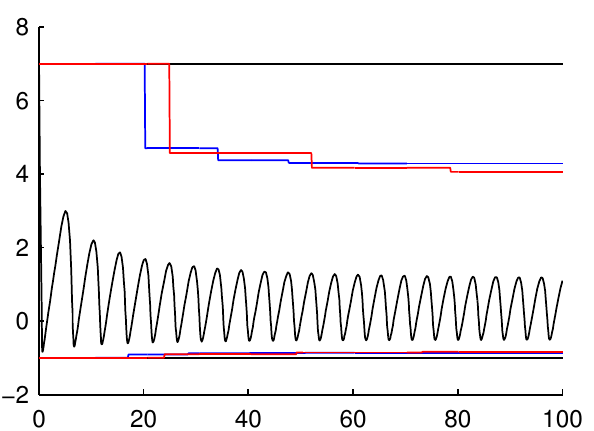}}
\put(-27,0){{\footnotesize $t$}}
\put(-125,50){{\footnotesize $u(t)$}}
\put(-40,109){{\footnotesize $N_0$}}
\put(-40,15){{\footnotesize $M_0$}}
\put(-40,80){{\footnotesize\color{blue}$N_{1,n}$}}
\put(-40,68){{\footnotesize\color{red}$N_{2,n}$}}
\subfigure[$\dot u(t)=-0.25 u(t)+\sigma u(t-1-u(t))$]{\includegraphics[scale=1]{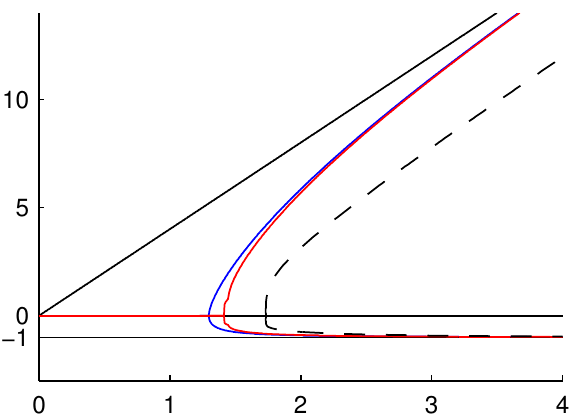}}
\put(-27,0){{\footnotesize $-\sigma$}}
\put(-136,51){{\footnotesize $N_0$}}
\put(-135,17){{\footnotesize $M_0$}}
\put(-40,74){{\footnotesize$\underset{t}{\max}\;u(t)$}}
\put(-120,40){{\footnotesize\color{blue}$N_{1,\infty}$}}
\put(-77,65){{\footnotesize\color{red}$N_{2,\infty}$}}

\subfigure[$\dot u(t)=-2u(t)-2.8u(t-1-u(t))$]{\includegraphics[scale=1]{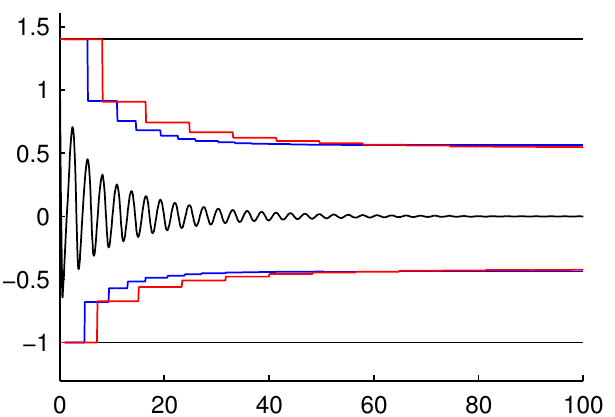}}
\put(-27,0){{\footnotesize $t$}}
\put(-60,60){{\footnotesize $u(t)$}}
\put(-40,111){{\footnotesize $N_0$}}
\put(-40,15){{\footnotesize $M_0$}}
\put(-140,76){{\footnotesize\color{blue}$N_{1,n}$}}
\put(-120,85){{\footnotesize\color{red}$N_{2,n}$}}
\put(-120,44){{\footnotesize\color{blue}$M_{1,n}$}}
\put(-135,30){{\footnotesize\color{red}$M_{2,n}$}}
\subfigure[$\dot u(t)=-2u(t)+\sigma u(t-1-u(t))$]{\includegraphics[scale=1]{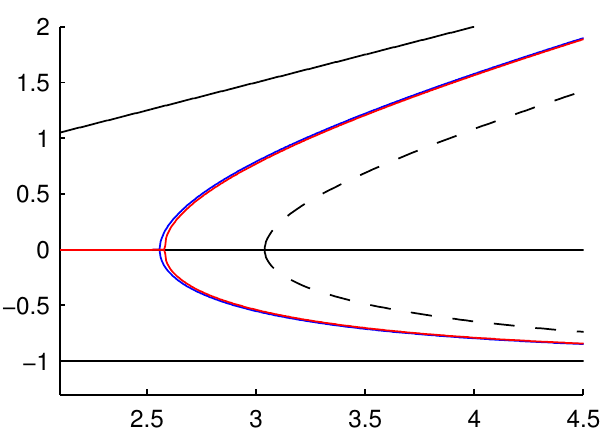}}
\put(-30,0){{\footnotesize $-\sigma$}}
\put(-125,100){{\footnotesize $N_0$}}
\put(-125,14){{\footnotesize $M_0$}}
\put(-40,80){{\footnotesize$\underset{t}{\max}\;u(t)$}}
\put(-40,37){{\footnotesize$\underset{t}{\min}\;u(t)$}}
\put(-125,75){{\footnotesize\color{blue}$N_{1,\infty}$}}
\put(-113,65){{\footnotesize\color{red}$N_{2,\infty}$}}
\put(-128,35){{\footnotesize\color{blue}$M_{1,\infty}$}}
\put(-117,40){{\footnotesize\color{red}$M_{2,\infty}$}}
\caption{\label{Fig:Bounds}(a) \& (c): Numerically simulated solution $u(t)$ to \eqref{Eqn:1Delay}
with $\phi(t)=0.99N_0$ for all $t\leq0$, along with the bounds $M_{k,n}$ and $N_{k,n}$ from Theorem~\ref{Thm:1Delay_changingbounds} with $k=1$ and $k=2$. In (a), $(\mu,\sigma)=(-0.25,-1.75)\not\in\wedg$
and the steady state solution is unstable. The solution $u(t)$ converges to a periodic orbit.
In (c), $(\mu,\sigma)=(-2,-2.8)\in\wedg$ and $u(t)$ converges to the steady state.
(b) \& (d): For the same values of $\mu$ as in (a) and (c) we vary $\sigma$ and plot the bounds
$[M_{k,\infty},N_{k,\infty}]$ on the persistent dynamics from Theorem~\ref{Thm:GAS} for $k=1$ and $k=2$
along with $[\min_t u(t),\max_t u(t)]$ for the periodic orbit created at the Hopf bifurcation
when $\sigma$ crosses the boundary of $\wedg$.}
\end{center}
\end{figure}

Theorem~\ref{Thm:1Delay_changingbounds} can be applied to obtain useful bounds
on solutions when Theorem~\ref{Thm:GAS} does not apply, including when the steady-state
is unstable.
In Figure~\ref{Fig:Bounds}(a) we present an example solution to \eqref{Eqn:1Delay} with $\mu$ and $\sigma$
both negative but outside $\Sigma_\star$ so the steady state is unstable. In that case
$u(t)$ approaches a stable periodic orbit. The bounds $M_{k,n}$ and $N_{k,n}$ for $k=1$ and $2$ are shown
and provide upper and lower bounds on the periodic solution (and any other persistent dynamics) which are
much better than the initial bounds $[M_0,N_0]$.

In Figure~\ref{Fig:Bounds}(c) we show an example solution to \eqref{Eqn:1Delay} with $(\mu,\sigma)\in\wedg$
so that the steady state is asymptotically stable, but we choose $(\mu,\sigma)$ outside the part of
$\wedg$ in which we can show either local or global asymptotic stability using Lyapunov-Razumikhin techniques.
In this case $u(t)$ approaches the stable steady state, and even
though the bounds $M_{k,n}$ and $N_{k,n}$ do converge to the steady state as $n\to\infty$ they do constrain
any possible persistent dynamics to a small part of the interval $[M_0,N_0]$.

It follows from Lemmas~\ref{Lem:H1} and~\ref{Lem:H2} that $M_{1,n}\leq M_{2,n}\leq0$ and $N_{1,n}\geq N_{2,n}\geq0$
so that $k=2$ gives tighter bounds than $k=1$ for fixed $n$ and as $n\to\infty$, and hence as $t\to\infty$.
However, because we require $T_{k,n+1}\geq T_{k,n}+(k+1)\tau_{k,n}$, at least for small $n$ we typically have
$T_{1,n}< T_{2,n}$ so the $k=1$ bounds apply for smaller values of $t$, as seen in Figure~\ref{Fig:Bounds}(c).

In Figures~\ref{Fig:Bounds}(b) and (d) we show how the bounds $M_{k,\infty}$ and $N_{k,\infty}$ vary
as $\sigma$ is varied for fixed $\mu$. We obtain $M_{k,\infty}=N_{k,\infty}=0$ for $(\mu,\sigma)$ values
which are above the blue curve in Figure~\ref{Fig:GAS}(a) for $k=1$ and above the red curve for $k=2$.
For other values of $(\mu,\sigma)\in\wedg$ the interval $[M_{k,\infty},N_{k,\infty}]$ is non-trivial even
though the steady-state is stable. At the boundary of $\wedg$ a Hopf bifurcation occurs, and we show
$\min_t\{u(t)\}$ and $\max_t\{u(t)\}$ for the resulting periodic solution. Our bounds display similar to
the actual solutions but are shifted slightly to the right on these graphs.
The resulting gap between the derived bounds $[M_{k,\infty},N_{k,\infty}]$ and the actual behaviour
of the solutions is caused by the use of the $H_{(k)}(v,x,y)$ functions in our bounds.  


\begin{figure}[ht]
\begin{center}
\subfigure[Lower bounds using $k=1$]{\includegraphics[width=2.5in]{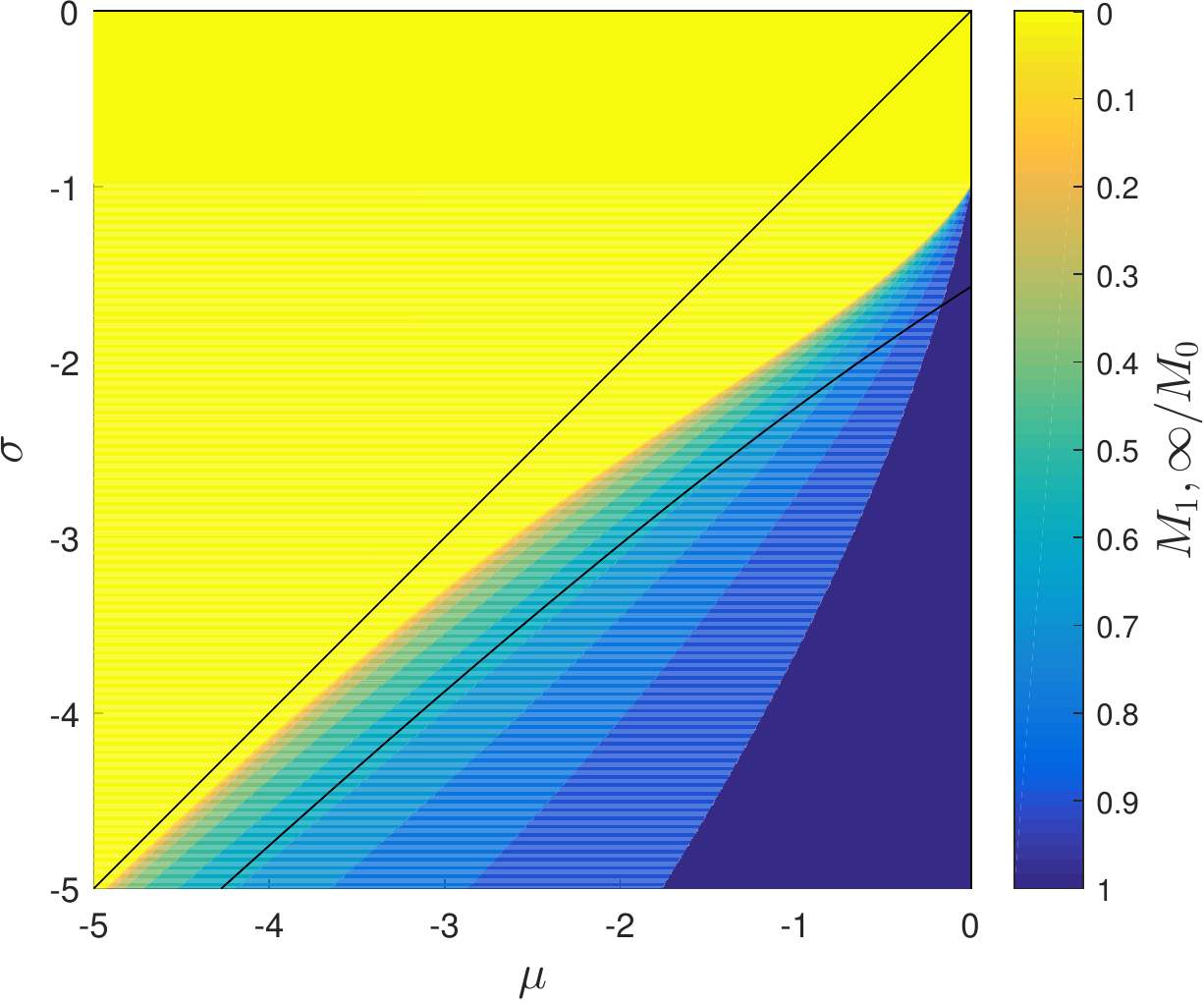}} \hspace{0.1in}
\subfigure[Upper bounds using $k=1$]{\includegraphics[width=2.5in]{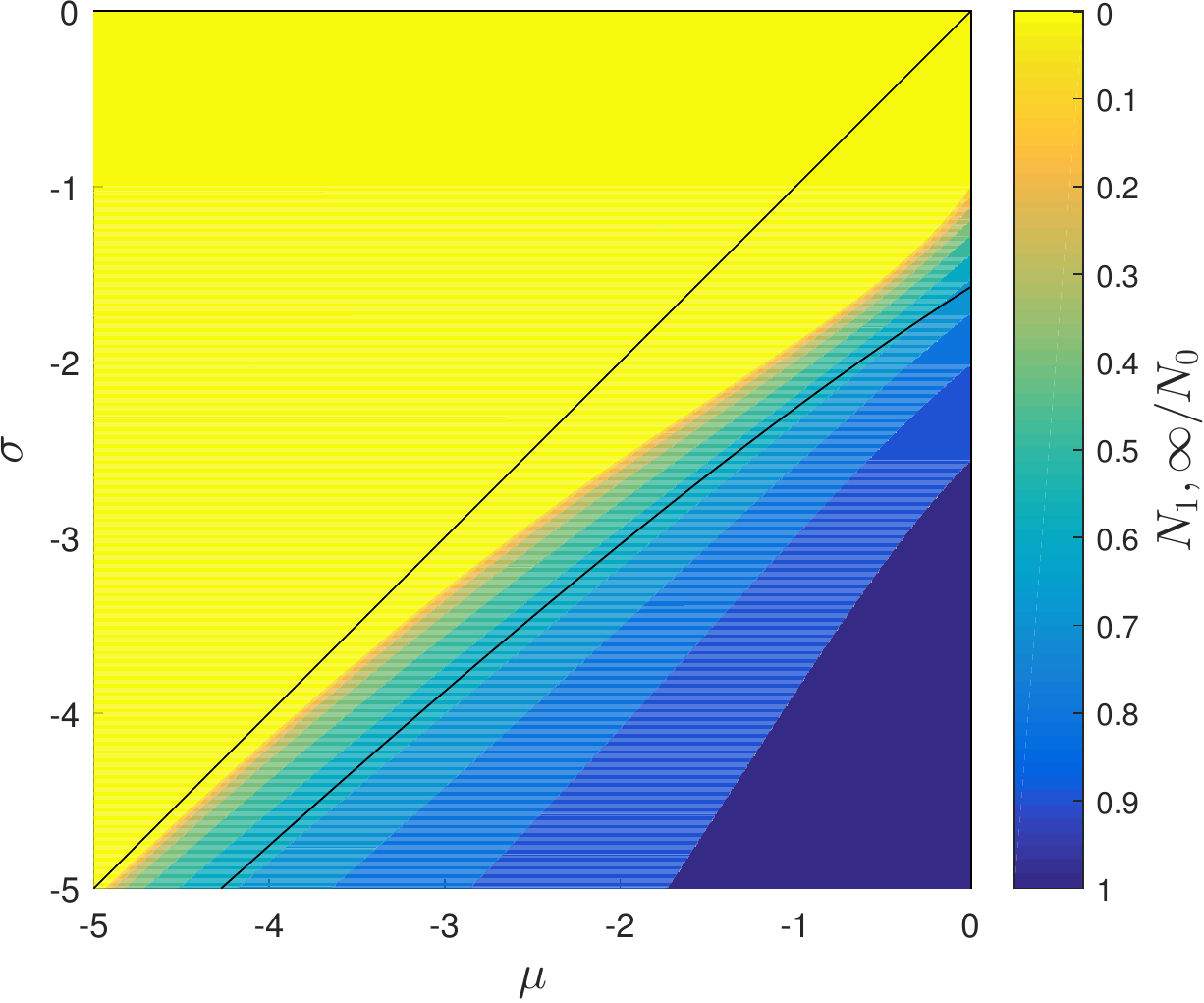}}
\vspace{0.2in}
\subfigure[Lower bounds using $k=2$]{\includegraphics[width=2.5in]{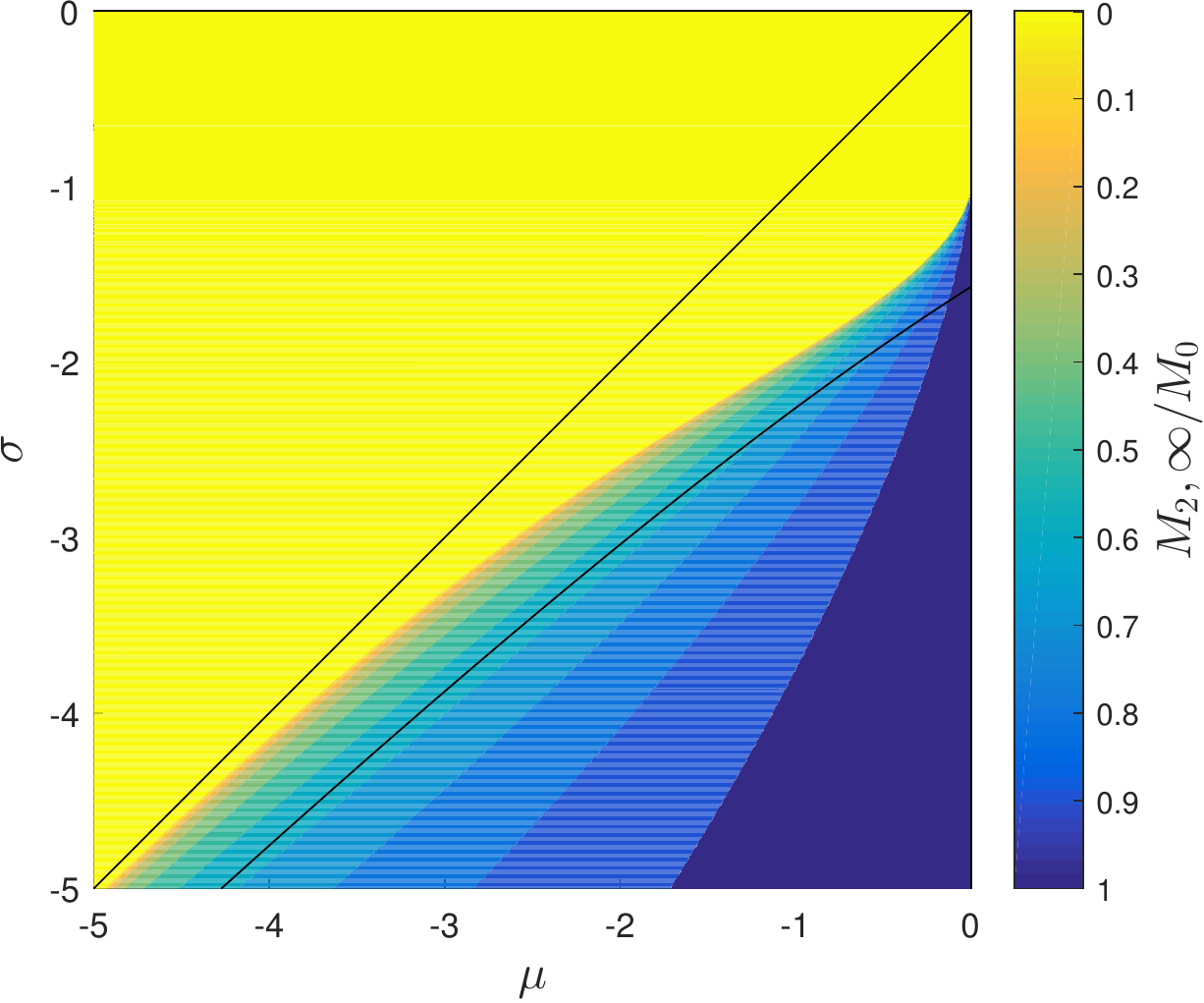}} \hspace{0.1in}
\subfigure[Upper bounds using $k=2$]{\includegraphics[width=2.5in]{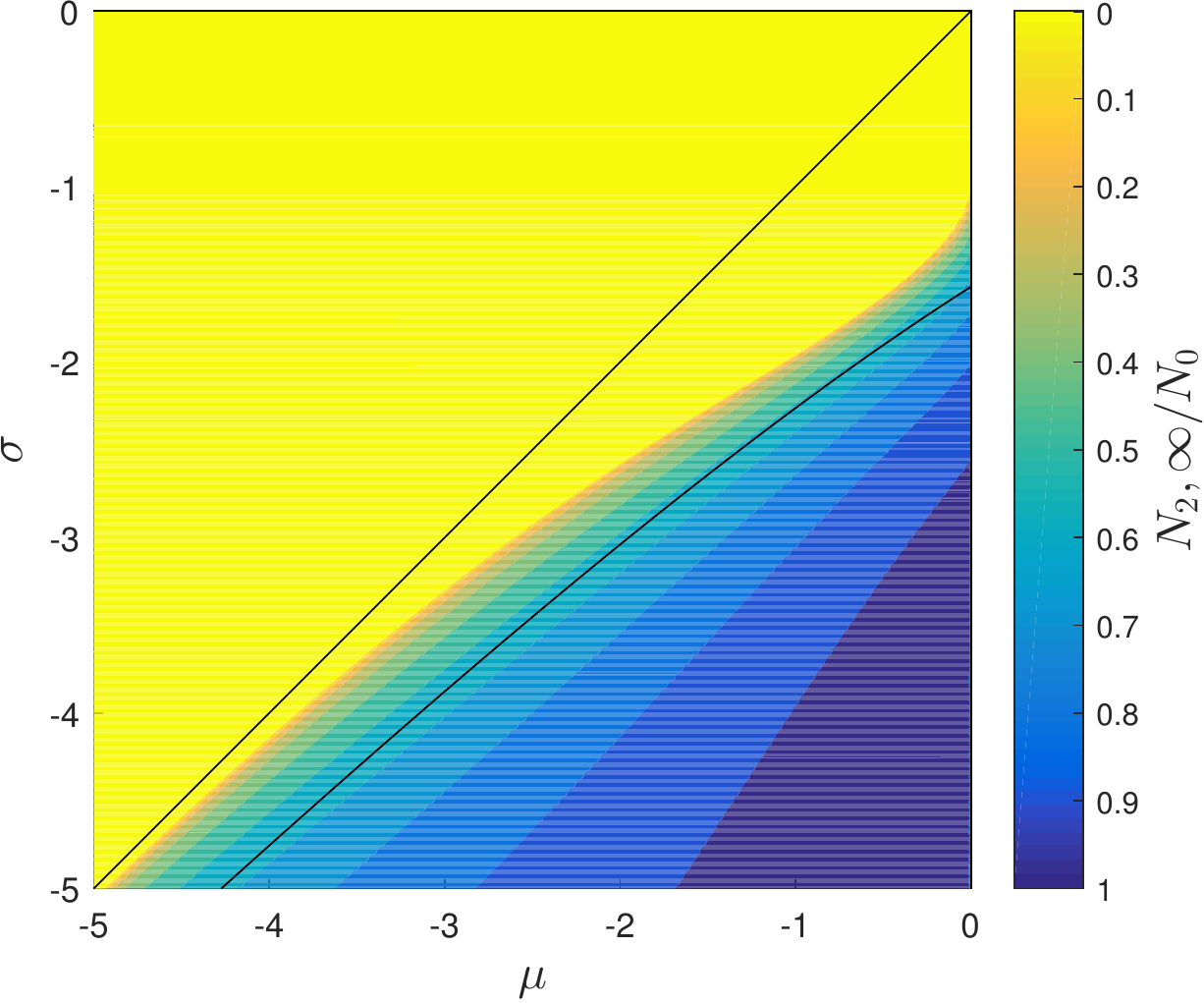}}
\caption{\label{Fig:BoundsPlane}Relative improvement in the lower and upper bounds of periodic solutions to \eqref{Eqn:1Delay} using the $k=1$ and $k=2$ Lyapunov-Razumikhin technique. The yellow region matches the region of global asymptotic stability indicated by the blue (for $k=1$) and red (for $k=2$) curves in Figure~\ref{Fig:GAS}.}
\end{center}
\end{figure}

In Figure~\ref{Fig:BoundsPlane} we present the improvement in the lower bounds ($M_{k,\infty}$) and upper bounds ($N_{k,\infty}$) of periodic solutions and other persistent dynamics of \eqref{Eqn:1Delay} relative to the initial lower bound $L_0$ and upper bound $M_0$, respectively, for $\mu<0$, $\lambda<0$. Results are shown for both $k=1$ and $k=2$.
The figure was drawn by calculating $M_{k,\infty}$ and $N_{k,\infty}$ at points in a $1000\times1000$ grid using equation \eqref{Eqn:Bifurcation_1Delayinfinity} and Matlab's nonlinear least squares solver.
The yellow region shows when $M_{k,\infty}=N_{k,\infty}=0$ which from Theorem~\ref{Thm:1Delay_changingbounds} implies global asymptotic stability of the steady-state solution.
This matches the region carved out by the blue (for $k=1$) and red (for $k=2$) curves in Figure~\ref{Fig:GAS}.

We note that within $\wedg$ for $\mu\ll0$ even though it has only been possible so far to prove asymptotic stability of the steady state by Lyapunov-Razumikhin techniques for a very thin sliver of parameter values close to the cone $\cone$, our generalised Lyapunov-Razumikhin techniques constrain the dynamics significantly with 
$|N_{k,\infty}-M_{k,\infty}|<0.6|N_0-M_0|$ across the entire width of the wedge $\wedg$ for $\mu\ll0$.
Near to the corner of $\wedg$, $(\mu,\sigma)=(0,-\pi/2)$, our generalised Lyapunov-Razumikhin techniques do not lead to a significant improvement on the lower bound with $M_{k,\infty}/M_0\approx1$, but there is a significant reduction in the upper bound with $N_{k,\infty}/N_0\approx0.6$. Below the wedge, where the steady state is unstable, the bounds constrain the amplitude of periodic orbits born in the Hopf bifurcation at the stability boundary, as already seen in Figure~\ref{Fig:Bounds}.

\section{Conclusions}
\label{Secn:Conclusions}
In this paper we generalised Lyapunov-Razumikhin techniques by considering solutions at local extrema rather than solutions about to escape some ball. This allowed us to derive sufficient conditions for global asymptotic stability of the steady-state of the model SDDE~\eqref{Eqn:1Delay},
expanding upon our previous work~\cite{MH:2016} where we derived Lyapunov stability and local asymptotic stability results for SDDEs using Lyapunov-Razumikhin techniques.

The required conditions for global asymptotic stability are easy to numerically verify for a given set of parameter values. The results are summarized in Theorem~\ref{Thm:cone} for $(\mu,\sigma)\in\cone$ and Theorem~\ref{Thm:GAS} for $(\mu,\sigma)\in\wedg$.
The region of parameter space where global asymptotic stability is assured is shown in Figure~\ref{Fig:GAS}(a).

It has already been shown that the zero solution for the state-dependent case ($c\ne 0$) of \eqref{Eqn:1Delay} is locally exponentially stable in the same parameter region $\Sigma_\star$ as that for the constant delay case ($c=0$)~\cite{GyoriHartung:1}.
In this paper we showed numerically (Figure~\ref{Fig:Bounds}) that in part of $\Sigma_\star$ the zero solution is not globally asymptotically stable, and our generalised Lyapunov-Razumikhin techniques enabled us to derive bounds on the persistent dynamics in that case.
Theorem~\ref{Thm:1Delay_changingbounds} provides bounds on solutions which can be used to bound periodic orbits and all other recurrent dynamics when the steady state is unstable, and also when the steady state is asymptotically stable, but Lyapunov-Razumikhin techniques are not able to establish global asymptotic stability.
An illustration of this technique and examples of the bounds we obtain on the persistent dynamics of \eqref{Eqn:1Delay} are shown in Figure~\ref{Fig:Bounds}.
The relative improvement in the lower and upper bounds of periodic solutions derived using the generalised Lyapunov-Razumikhin techniques of Theorem~\ref{Thm:1Delay_changingbounds} are shown in Figure~\ref{Fig:BoundsPlane}.


In this paper we focused on deriving bounds for the solutions of the model SDDE \eqref{Eqn:1Delay}, but this method can be extended to other scalar SDDE problems. In particular, our methods would apply directly to problems with delayed negative-feedback components of the following form,
$$\dot u(t) = R \bigl(u(t)\bigr) + S \bigl(  u(t-\tau(t,u(t)) \bigr),$$
where $S '(u)<0$. It is also possible to apply these techniques to problems with multiple delays, as was considered briefly in \cite{Humphries:1}. The difficulty in the case of multiple delays is that the analogue of \eqref{Eqn:Bounds:EtaConditionA} would only give a constraint on the delayed values of the solution, rather than giving the value directly, but we were already able to work with constraints on the delay in \cite{MH:2016} where we considered
$\tfrac{d}{dt}\|u(t)\|\geq0$ on the boundary of a ball, rather than $\udot(t)=0$ at a local extremum, so such difficulties are surmountable.


\section*{Acknowledgments}
A.R. Humphries is grateful to Tibor Krisztin, John Mallet-Paret, Roger Nussbaum and Hans-Otto Walther
for productive discussions and suggestions, and to NSERC (Canada) for funding
through the Discovery Award program. F.M.G. Magpantay is grateful to Jianhong Wu for helpful discussions, and to McGill University, York University, The Institut des Sciences Math\'ematiques (Montreal, Canada) and NSERC (Canada) for funding.

\def\bibsection{\section*{References}}





\end{document}